\pdfoutput=1
\RequirePackage{ifpdf}
\ifpdf 
\documentclass[pdftex]{sigma}
\else
\documentclass{sigma}
\fi

\numberwithin{equation}{section}

\newtheorem{thm}{Theorem}[section]
\newtheorem*{thm*}{Theorem}
\newtheorem{cor}[thm]{Corollary}
\newtheorem{lem}[thm]{Lemma}
\newtheorem{prop}[thm]{Proposition}

\theoremstyle{definition}
\newtheorem{defin}[thm]{Definition}

\newtheorem{rem}[thm]{Remark}
\newtheorem{Question}[thm]{Question}

\def\C{{\bf C}}
\def\N{{\bf N}}

\def\R{{\bf R}}

\def\V{{\bf V}}

\def\W{{\bf W}}

\newcommand{\cV}{{\mathcal V}}

\newcommand{\cS}{{\mathcal S}}
\newcommand{\BR}{{\bf R}}

\newcommand{\lieg}{\mathfrak{g}}
\newcommand{\lieh}{\mathfrak{h}}
\def\g{{\mathfrak g}}
\def\h{{\mathfrak h}}

\def\p{{\mathfrak p}}
\def\q{{\mathfrak q}}
\def\a{{\mathfrak a}}
\def\b{{\mathfrak b}}

\def\n{{\mathfrak n}}

\def\lieo{{\mathfrak o}}
\newcommand{\so}{\mathfrak{so}}

\newcommand{\Aut}{\operatorname{Aut}}
\newcommand{\Ad}{\operatorname{Ad}}
\newcommand{\ad}{\operatorname{ad}}
\newcommand{\GL}{{\operatorname{GL}}}
\newcommand{\SU}{{\rm SU}}
\newcommand{\SO}{{\rm SO}}

\begin{document}

\allowdisplaybreaks

\newcommand{\arXivNumber}{2505.14665}

\renewcommand{\PaperNumber}{061}

\FirstPageHeading

\ShortArticleName{Tractor Embedding Theorem with Conformal Application}

\ArticleName{An Embedding Theorem for Tractor Bundles,\\ and an Application in Conformal Pseudo-Riemannian\\ Geometry}

\Author{Karin MELNICK~$^{\rm a}$ and Katharina NEUSSER~$^{\rm b}$}

\AuthorNameForHeading{K.~Melnick and K.~Neusser}

\Address{$^{\rm a)}$~Department of Mathematics, University of Luxembourg,\\
\hphantom{$^{\rm a)}$}~6 avenue de la Fonte, L-4364 Esch-sur-Alzette, Luxembourg}
\EmailD{\mail{karin.melnick@uni.lu}}
\URLaddressD{\url{https://math.uni.lu/~melnick/}}

\Address{$^{\rm b)}$~Department of Mathematics and Statistics, Masaryk University,\\
\hphantom{$^{\rm b)}$}~Kotla\v rsk\'a 267/2, 611 37 Brno, Czech Republic}
\EmailD{\mail{neusser@math.muni.cz}}
\URLaddressD{\url{https://is.muni.cz/www/neusser/}}

\ArticleDates{Received October 14, 2025, in final form June 04, 2026; Published online June 22, 2026}

\Abstract{We provide an extension of the Gromov--Zimmer embedding theorem for Cartan geometries of [Bader U., Frances C., Melnick K., \textit{Geom. Funct. Anal.} \textbf{19} (2009), 333--355, arXiv:0709.3844] to tractor bundles carrying any invariant connection, including tractor connections and prolongation connections of first BGG operators for parabolic geometries. As an application, we prove a rigidity result for conformal actions of special pseudo-unitary groups on closed, simply connected, analytic pseudo-Riemannian manifolds.}

\Keywords{Cartan geometries; parabolic geometries; conformal manifolds; automorphisms of geometric structures; pseudo-Riemannian geometry; simple Lie transformation groups}

\Classification{53C15; 53C18; 53C24; 22F50; 57S20}

\section{Introduction}
A Cartan geometry is a manifold infinitesimally modeled on a
homogeneous space $G/P$, where $G$ is a Lie group and $P$ is a closed
subgroup. It comprises a principal $P$-bundle over the manifold,
equipped with a \emph{Cartan connection} taking values in $\lieg$; see Definition~\ref{defin.cartan.geom} below. Essentially all classical
rigid differential-geometric structures correspond canonically to
Cartan geometries for a suitable homogeneous model, in such a way that the
automorphisms of the structure are the same as the automorphisms of
the Cartan geometry.

Our results concern Cartan geometries, particularly parabolic geometries, having large
automorphism group. This property typically implies further special properties of the geometry, for example,
the existence of solutions to natural overdetermined geometric PDEs. Conversely,
when there are many solutions to such PDEs, one may expect a large automorphism group. Either property implies restrictions on the curvature. The results of this paper are intended to contribute to the concrete understanding of this interplay.

\subsection{Parabolic geometries and tractor bundles}
\label{sec.intro.tractor}

A \emph{parabolic geometry} is a Cartan geometry infinitesimally modeled
on a homogeneous space~$G/P$, where $G$ is a semisimple Lie group and
$P < G$ a parabolic subgroup. This rich family of geometries includes
conformal and projective structures, along with many others. In general, a parabolic
geometry on a manifold $M$ corresponds to a filtration of the tangent
bundle, $TM$, with some additional geometric structures on the
subquotients of the filtration.

Tractor bundles are vector bundles {$\mathcal{V} = \hat{M} \times_\rho {\bf V}$} associated to the Cartan bundle
$\hat{M}$, for $G$-representations $\rho$ restricted to $P$, and equipped
with certain geometrically significant, automorphism-invariant connections.
On the one hand, any tractor bundle is equipped with the tractor connection $\nabla^\omega$, induced by the Cartan connection $\omega$,
and on the other hand with the prolongation connection $\nabla^{\mathcal V}$ of the \emph{first BGG-operator} corresponding to $\mathcal V$, which
in general differs from $\nabla^\omega$ by curvature terms; see Section~\ref{sec_BGG} for details. Parallel sections for these connections correspond to solutions of overdetermined
geometric PDEs
on the manifold; they form a finite-dimensional vector space.

The conformal class of a metric of signature $(p,q)$ on a manifold
$M$, with $p+q = \mbox{dim } M \geq 3$, corresponds to a Cartan geometry
modeled on $G/P$ for $G = \mbox{PO}(p+1,q+1)$ and $P$ the maximal
parabolic subgroup stabilizing an isotropic line in $\BR^{p+1,q+1}$.
The homogeneous space $G/P$ is called the \emph{M\"obius space} of
type $(p,q)$, denoted $\mbox{\bf M\"ob}^{p,q}$, and is diffeomorphic to a two-fold quotient of ${\bf S}^p \times {\bf S}^q$. It carries a
conformally $G$-invariant, conformally flat metric of signature~$(p,q)$. The Cartan
bundle is a principal $P$-subbundle of the second-order frame bundle
of~$M$. The~Cartan connection $\omega$ has been a valuable tool
in recent progress on rigidity of conformal transformations of
semi-Riemannian manifolds (see, for example, \cite{me.notices} for an
overview). Vanishing of the Cartan curvature of this connection over an
open subset $U \subseteq M$ is equivalent to conformal flatness of the
metric on $U$.

A semi-Riemannian metric $g$ on $M^n$, $n \geq 3$, is \emph{Einstein}
if \smash{$\operatorname{Ric} g = \frac{S}{2} g$}, where $\operatorname{Ric}$ is the Ricci
curvature tensor and $S$ is the scalar curvature.
It is \emph{conformally almost-Einstein} if there exists a
smooth function $\sigma$, not constant $0$, and a smooth function $\nu$ such that
\begin{equation}
\label{eqn.conf.einstein}
 \nabla^2 \sigma + \sigma \cdot P = \nu \cdot g,
 \end{equation}
 where $\nabla^2$ is the Hessian, and
 \[
 P = \frac{1}{n-2} \left( \operatorname{Ric} - \frac{S}{2(n-1)} \cdot g \right)
 \]
is the \emph{Schouten tensor}; in other words, $\nabla^2 \sigma +
\sigma \cdot P$ is pure trace. When $\sigma$ is nonvanishing,
then~${\sigma^{-2} \cdot g}$ is an Einstein metric in the conformal class of $g$
(see, e.g., \cite[Proposition~3.6]{curry.gover.tractors}).

Suppose $\hat{M} \rightarrow M$ is the Cartan bundle of a conformal
semi-Riemannian structure. Let ${\bf V} = \BR^{p+1,q+1}$
and $\rho$ the standard representation of ${\rm O}(p+1,q+1)$. The
\emph{(standard) conformal tractor bundle} of such a structure is the associated
vector bundle $\mathcal{V} = \hat{M} \times_\rho {\bf V}$ and it is equipped with the $\omega$-induced tractor connection $\nabla^{\omega}$.
Sections of
$\mathcal{V}$ yield, via projection to a certain line bundle
and a choice of metric in the conformal class,
functions on $M$; functions coming from parallel sections for
$\nabla^{\mathcal{\omega}}$ correspond to
solutions of the conformal almost-Einstein equation. In~this case, the
tractor connection equals the prolongation connection $\nabla^{\mathcal V}$ of the first BGG operator.
See Section~\ref{sec.conf.tractor} below, and the references \cite{cap.gover.tractor.calc, curry.gover.tractors} for more details.

 There are many other significant geometric PDEs on conformal manifolds, and more generally on manifolds equipped with parabolic geometries, whose solutions correspond under prolongation
 to parallel sections of the prolongation connection of a first BGG-operator; on conformal manifolds this applies for instance also to the equations for conformal Killing forms \cite{GoverSilhan, Semmelmann} and for twistor spinors \cite{BFGK}, and for projective and
 c-projective structures to the equation controlling the metrizability of these structures \cite{cemn.cproj, eastwood.matveev.proj.tractor}, among many others.

Infinitesimal automorphisms of geometric structures underlying Cartan geometries also belong to the framework of parallel
sections of natural connections on tractor bundles. Let $\bigl(M, \hat{M}, \omega\bigr)$ be any Cartan geometry modeled on
$(\lieg,P)$. The \emph{adjoint tractor bundle} is the associated vector
bundle $\mathcal{A}$ for ${\bf V} = \lieg$ with the adjoint
representation. A section of $\mathcal{A}$ gives rise via the Cartan
connection to a vector field on $\hat{M}$ which is invariant under the
right-$P$-action, and thus descends to a vector field on $M$.
\v{C}ap observed that for a certain connection on~$\mathcal{A}$,
again related to an intrinsically defined differential operator on $M$, parallel sections correspond to
infinitesimal automorphisms of the Cartan geometry -- that is, vector
fields for which the flow preserves the Cartan connection~\cite{cap.deformations}. These
descend to vector fields on $M$ for which the flow preserves the
geometric structure encoded by the Cartan geometry, called
\emph{infinitesimal automorphisms} or \emph{Killing fields} of the structure.

This article concerns a
generalization of the Gromov--Zimmer embedding theorem, which originally pertains to
infinitesimal automorphisms, to general parallel sections for automorphism-invariant connections of tractor bundles.
We briefly recall the embedding theorem for infinitesimal automorphisms.

\subsection{Gromov--Zimmer embedding theorem}
\label{sec.intro.embedding}

Zimmer's embedding theorem first appeared in \cite{zimmer.lorentz} in the context of a noncompact simple Lie group $H$ acting by automorphisms of a $G$-structure $\omega$ of algebraic type on a compact manifold~$M^n$, such that $G$ also defines a volume form. In this context, the theorem gives a linear injection $\iota\colon \lieh \rightarrow \BR^n$ and an algebraic group $\check{H} < G$ locally isomorphic to $H$ such that $\iota$ intertwines the adjoint representation ${\rm Ad}_\lieh H$ with the representation of $\check{H}$ on $\iota(\lieh) \subseteq \BR^n$; in particular,~$H$~locally embeds in the structure group $G$. Zimmer provided an impressive application for~$\omega$ a~Lorentzian metric, proving that $H$ must be locally isomorphic to ${\rm SL}(2,\BR)$.

Gromov subsequently gave a very general version of the embedding theorem in \cite[Theorem~5.2.A]{gromov.rgs}, which applies to a Lie algebra $\lieh$ of infinitesimal automorphisms of a Gromov-rigid geometric structure $\omega$ of algebraic type on a compact manifold $M$, together with a group $B \leq \Aut(M,\omega)$ preserving a finite measure $\mu$ on $M$ and normalizing $\lieh$. Combined with Gromov's Frobenius theorem \cite[Theorem~1.3.A]{gromov.rgs}, the embedding theorem led to the Gromov--Zimmer centralizer and representation theorems (see \cite{zimmer.gromovrep}).

Our version of the embedding theorem will be a direct generalization of the following version proved by Bader, Frances, and the first author in \cite[Theorems~1.2 and 4.1]{bfm.zimemb} for Cartan geometries.

\begin{thm}
Let $\bigl(M, \hat{M}, \omega\bigr)$ be a Cartan geometry modeled on $(\lieg,P)$ and assume that ${\rm Ad}_{\lieg} P<\operatorname{Aut} \lieg$ is Zariski closed. Let $H\leq\operatorname{Aut}(M,\omega)$ be a Lie subgroup. Let $B \leq H$ preserve a probability measure $\mu$ on $M$. Then for $\mu$-almost-every $x\in M$, for all $\hat x\in\hat M_x$, there exists an algebraic subgroup $\check B<P$ and an algebraic epimorphism
$R\colon \check{B} \rightarrow \bar{B}_d$ such that
\[ (\Ad g) (\iota_{\hat{x}} (X)) = \iota_{\hat{x}} ( R(g). X ) \qquad \forall X \in \mathcal{\lieh},\, g \in \check{B},\]
where $\iota_{\hat{x}} \colon \lieh \rightarrow \lieg$ is the linear injection $X \mapsto \omega_{\hat{x}}(X_{\hat{x}}).$
\end{thm}
\noindent The group $\bar{B}_d$ is the discompact radical of ${\rm Ad}_{\lieh} B \leq \operatorname{Aut} \lieh$, see Definition~\ref{def.discompact.rad} below. Injectivity of $\iota_{\hat{x}}$ is evident from the axioms for the Cartan connection $\omega$, see Definition~\ref{defin.cartan.geom} below.

This embedding theorem says that for $\mu$-almost-every $x \in M$, for all $\hat{x} \in \hat{M}_x$, the injection $\iota_{\hat{x}}$ intertwines the $B$-representation on $\lieh$ with that of an algebraic subgroup $\check{B} < P$ on $\iota_{\hat{x}}(\lieh) \subseteq \lieg$.

\subsection{Results}

Here is our tractor generalization of the embedding theorem. By invariant connection is meant one which is invariant by all automorphisms; in practice, such connections will be defined in a~manner intrinsic to the geometry and invariant for this reason.

\begin{thm}
\label{thm.embedding-0}
 Let $\bigl(M, \hat{M}, \omega\bigr)$ be a Cartan geometry modeled on $(\lieg,P)$. Let $\rho$ be an infinitesimally faithful, completely reducible $G$-representation on $\V$ with $\rho(P)$ algebraic, with associated tractor bundle $\mathcal{V}$, equipped with any invariant connection $\nabla$.
 Let $B \leq \operatorname{Aut}(M,\omega)$ be a Lie subgroup preserving a
 probability measure $\mu$ on $M$.

Let $\mathcal{S}$ denote the $\nabla$-parallel sections of $\cV$.
Then for $\mu$-almost-every $x \in M$, for all $\hat{x} \in \hat{M}_x$, there exists
an algebraic subgroup $\check{B} <P$ with an algebraic epimorphism $R\colon \check{B} \rightarrow \bar{B}_d$ such that
\[ g .\iota_{\hat{x}} (X) = \iota_{\hat{x}} \left( R(g). X \right) \qquad \forall X \in \mathcal{S},\, g \in \check{B}.\]

\end{thm}
\noindent The linear injection $\iota_{\hat{x}}\colon \Gamma(\mathcal{V}) \rightarrow {\bf V}$ is determined by $\hat{x} \in \hat{M}$ as described in Section~\ref{sec.parabolic.tractor}, and $\bar{B}_d$ denotes the discompact radical of the image of $B$ in $\GL(\mathcal{S})$ (see Definition~\ref{def.discompact.rad} below).

As an application, we prove the following rigidity theorem in conformal pseudo-Riemannian geometry.

\begin{thm}
\label{thm.conf.application}
Let $(M,g)$ be a closed pseudo-Riemannian manifold of type $(p,q)$ with $p\leq q$ and $p+q = \dim M \geq 3$. Suppose that a group $H$ locally isomorphic to $\SU(p',q')$, with $p'<q'$, acts conformally on $(M,g)$. Then $p \geq 2p'-1$. If $p=2p'-1$, then necessarily $p<q$; if $(M,g)$ is moreover real-analytic and simply connected, then it is conformally equivalent to the universal cover of $\mbox{{\bf M\"ob}}^{p,q}$, with $H$ acting via a homomorphism to ${\rm O}(p+1,q+1)$.
\end{thm}

Without assuming $(M,g)$ simply connected or real-analytic, Pecastaing proved this result for $p'=1$ \cite[Corollary~1]{pecastaing.smooth.sl2r}. He proved in \cite[Theorem~3]{pecastaing.rank1} that any simple group $H < \mbox{Conf}(M,g)$, for $M$ closed and $g$ Lorentzian, locally embeds in ${\rm O}(2,n+1)$. In analogy with Zimmer's result on volume-preserving actions of simple groups
in Section~\ref{sec.intro.embedding} above, the following question arises.

\begin{Question}\label{quest.simple.embed}
Let $(M,\omega)$ be a closed manifold with a parabolic geometry modeled on $(\lieg, P)$. Suppose that $H \leq \Aut(M, \omega)$ is a simple Lie group. Does the Lie algebra $\lieh$ of $H$ embed in $\lieg$?
\end{Question}

For the proof of Theorem~\ref{thm.conf.application}, we apply the embedding theorem to two tractor bundles simultaneously: to the conformal Killing fields, as in the usual version of the theorem, and also to the conformal-to-almost-Einstein scales, corresponding to the parallel sections of the standard conformal tractor bundle.
From the first embedding, we prove the algebraic Proposition~\ref{prop.sutoso.roots}, which says $p \geq 2 p'-1$; in the case of equality, it gives $p< q$ as well as precise information about how the positive root spaces of $\lieh$ embed in those of $\lieg$. Proposition~\ref{prop.sutoso.roots} demonstrates how in an extremal case, and with $H$ simple, careful algebraic analysis of the output of the classical embedding theorem yields that it is very close to a Lie algebra embedding of $\lieh$ into $\lieg$.

Using the conclusions of Proposition~\ref{prop.sutoso.roots}, we obtain a special one-parameter subgroup in $\check{B}$, which is of \emph{balanced linear} type.
The Frobenius theorem ensures that this one-parameter subgroup is the isotropy of a one-parameter group of local conformal transformations fixing a point $x$ (see~Corollary~\ref{cor.emb.frob}); analyticity is used at this step. This conformal flow in turn implies the existence of local tractor solutions, more precisely of a 2-parameter family of Ricci-flat metrics in the conformal class restricted to a neighborhood of $x$. Analyticity of the metric and $1$-connectedness of~$M$ allow us to extend these solutions to all of~$M$.
The embedding theorem applied to the second bundle implies $H$ acts nontrivially on the space $\mathcal{S}$ of conformal-to-Einstein tractor solutions. Many more solutions ensue, enough to guarantee conformal flatness. The topological~assumptions on $M$ then imply
\[
(M,[g]) \cong \widetilde{\mbox{{\bf M\"ob}}}^{p,q}.
\]

If $H$ locally isomorphic to $\SU(p',q')$ acts isometrically on $(M,g)$ as above, then by Zimmer's result there is a Lie algebra embedding $\mathfrak{su}(p',q') \hookrightarrow \so(p,q)$, which implies $p \geq 2p'$ (see Proposition~\ref{prop.sutoso.roots}). Thus the extremal case $p = 2p'-1$ of Theorem~\ref{thm.conf.application} corresponds to \emph{essential} conformal transformations.

A version of Question \ref{quest.simple.embed} for conformal transformations, in particular for essential conformal groups, is discussed in \cite{pecastaing.rank1}. See \cite{bn.simpleconf,fz.simpleconf} for important results in this direction in terms of the real rank of $H$.

\section{Background}

\subsection{Parabolic geometries and tractor connections}

\label{sec.parabolic.tractor}
We will now briefly recall some basics about Cartan geometries, and parabolic geometries in particular.
For a general introduction to Cartan geometries, see \cite{sharpe}.
The comprehensive reference on
parabolic geometries is \cite{cap.slovak.book.vol1}.

\subsubsection{Cartan geometries and parabolic geometries}
Suppose $G$ is a Lie group and $P<G$ a closed subgroup. We will always assume that $P$ contains no nontrivial normal subgroup of $G$.
\begin{defin}
 \label{defin.cartan.geom}
 A \emph{Cartan geometry} modeled on $(\g, P)$ on a manifold $M$ is given by a principal $P$-bundle \smash{$\pi\colon \hat{M}\rightarrow M$} equipped with a
 \emph{Cartan connection}, that is, a $P$-equivariant $1$-form ${\omega\in\Omega^1\bigl(\hat{M},\g\bigr)}$ such that
 \begin{itemize}\itemsep=0pt
 \item $\omega_{\hat{x}}\colon T_{\hat{x}} \hat{M}\rightarrow \g$ is an isomorphism for all $\hat{x}\in \hat{M}$;
 \item $\omega(X)= \check{X}$ for any $\check{X}\in \p$,
 where $X_{\hat{x}}= \frac{{\rm d}}{{\rm d}t}\big|_{t=0 } \hat{x}\cdot {\rm e}^{t \check{X}}$ is the vertical vector field on~$\hat{M}$ generated by $\check{X}\in\p$.
 \end{itemize}
 \end{defin}
 The \emph{homogeneous model} of a Cartan geometry as in Definition~\ref{defin.cartan.geom} is the homogeneous space~$G/P$ equipped with the natural projection $G\rightarrow G/P$ and
 the Maurer--Cartan form ${\omega_G\in\Omega^1(G,\g)}$ of $G$, which encodes the natural trivialization $TG\cong G\times\g$ by left-invariant vector fields.
 The \emph{curvature of a Cartan geometry} is given by the following semibasic, $\g$-valued, $P$-equivariant $2$-form on $\hat{M}$:
 \begin{equation*}
 K(\xi,\eta)={\rm d}\omega(\xi,\eta)+[\omega(\xi), \omega(\eta)] \qquad\textrm{for } \xi,\eta\in\mathfrak X(\hat{M}).
 \end{equation*}
 The significance of the curvature is that it vanishes if and only if the Cartan geometry is locally isomorphic to its homogeneous model.

 Given a Cartan geometry, any representation $\rho\colon P\rightarrow\textrm{GL}({\bf V})$ of $P$ induces
 a vector bundle by forming the associated bundle to the Cartan bundle:
 \[
 \mathcal V =\hat{M} \times_\rho {\bf V}.
 \]
 Recall that the space of smooth sections $\Gamma(\mathcal V)$ of $\mathcal V$ can be identified with the space \smash{$C^\infty\bigl(\hat{M}, {\bf V}\bigr)^P$} of smooth, $P$-equivariant functions $f\colon \hat{M}\rightarrow {\bf V}$, that is,
 \[f(\hat{x}\cdot p)= {\rho\bigl(p^{-1}\bigr). f(\hat{x})}.\]
 In this way, each $\hat{x} \in \hat{M}$ gives a linear function $\iota_{\hat{x}}\colon \Gamma(\mathcal{V}) \rightarrow {\bf V}$ by evaluation.

 Sections of the adjoint bundle $\mathcal{A} = \hat{M} \times_P \lieg$ from Section~\ref{sec.intro.tractor} correspond via $\omega$ to right-$P$-invariant vector fields on $\hat{M}$, which descend to vector fields on $M$.
The Cartan connection $\omega$ induces isomorphisms of vector bundles
 \begin{equation}\label{iso_tangent_bundle}
 TM\cong \hat{M} \times_P \g/\p\qquad\textrm{and}\qquad T^*M\cong \hat{M} \times_P (\g/\p)^*.
 \end{equation}
 These identifications can easily be derived from the axioms for $\omega$ (see also \cite[Theorem~5.3.15]{sharpe}).
 Because the curvature $K$ is semibasic and $P$-equivariant, it
 descends to a $2$-form on $M$ with values in $\mathcal{A}$, which, via \eqref{iso_tangent_bundle}, can be identified with a $P$-equivariant \emph{Cartan curvature function}
 \begin{equation*}
 \kappa\colon\ \hat{M} \rightarrow\Lambda^2 (\g/\p)^*\otimes \g.
 \end{equation*}

Many important geometric structures admit equivalent descriptions as Cartan geometries, in particular semi-Riemannian metrics. A broad and interesting class of examples are the geometric structures underlying \emph{parabolic geometries}, which are Cartan geometries of type $(\g,P)$, where $P<G$ is a parabolic subgroup of a semisimple Lie group.
These correspond to certain bracket-generating distributions on $M$ endowed with a geometric structure defined by $G_0$, the reductive complement in a Levi decomposition of $P \cong G_0 \ltimes P_+$. These structures correspond canonically to \emph{normal} parabolic Cartan geometries. See \cite{cap.schichl.equiv,morimoto,tanaka} for the precise general statements, including the notions of normal and \emph{regular}.
In the table below, we list some of the most prominent examples of geometric structures admitting an equivalent description as regular normal parabolic geometries of type $(\g,P)$. For more examples, see \cite{cap.slovak.book.vol1}. A concise reference on the construction
of the Cartan connection in conformal or projective geometry is \cite[Section IV]{kobayashi.transf}. All parabolic geometries will be assumed to be regular and normal throughout the sequel.

\begin{table}[h]\centering
\small\renewcommand{\arraystretch}{1.15}
\begin{tabular}{ |l|c|l| }
\hline
 Geometric structure & $\g$ & $P$ \\
\hline\hline
Projective structures on $n$-manifolds & $\mathfrak{sl}(n+1,\BR)$& stabilizer of line in $\BR^{n+1}$\\
\hline
Almost c-projective structures &$\mathfrak{sl}(n+1,\C)$ &stabilizer of complex \\
on $2n$-manifolds, $n\geq 2$ &&line in $\C^{n+1}$\\
\hline
Conformal $n$-manifolds of &$\mathfrak{so}(p+1,q+1)$ & stabilizer of an isotropic \\
signature $(p,q)$, $n=p+q\geq 3$ & &line in $\BR^{p+1,q+1}$ \\
 \hline
 Almost quaternionic $4n$-manifolds& $\mathfrak{sl}(n+1,{\bf H})$ & stabilizer of a quaternionic \\
 &&line in ${\bf H}^{n+1}$ \\
 \hline
 Non-degenerate partially integrable &$\mathfrak{su}(p+1,q+1)$& stabilizer of an isotropic \\
 CR-structures of hypersurface type & &complex line in Hermitian \\
 of signature $(p,q)$ &&vector space $\C^{p+1,q+1}$\\
 \hline
 Generalized path geometries&$\mathfrak{sl}(n+2,\R)$& stabilizer in $\BR^{n+2}$ of a flag \\
 on $2n+1$-manifolds, $n\geq 2$ && of line inside a plane\\
 \hline
 (2,3,5)-distributions on $5$-manifolds &$G_2$&stabilizer of highest weight\\
 && line in the $7$-dimensional \\
 &&fundamental representation\\
 \hline
\end{tabular}
\end{table}

On the Lie algebra level, a Levi decomposition will be denoted $\p \cong \g_0 \ltimes \p_+$. There is a $G_0$-invariant subspace of $\g$ complementary to $\p$, which is denoted $\g_-$; it is of course isomorphic as a~$G_0$-module to $\g/\p$. The nilradical $\p_+ \lhd \p$ is moreover isomorphic as a $P$-module to $(\g/\p)^*$. Any representation $\bf V$ of $\p$ on which the center of $\g_0$ acts diagonalizably
has
 a $P$-invariant filtration
\begin{equation}\label{filtration_rep}
{\bf V}={\bf V}^{-\ell}\supset {\bf V}^{- \ell+1}\supset \cdots \supset {\bf V}^0\supset {\bf V}^{1}\supset \cdots \supset {\bf V}^{m} \supset 0
\end{equation}
such that each quotient ${\bf V}^{i}/{\bf V}^{i+1}$ is a completely reducible $\p$-representation, which implies that~$\p_+$ acts trivially on ${\bf V}^{i}/{\bf V}^{i+1}$.
In particular, the $\p$-representation $\g$ admits such a filtration
 \begin{equation*}
 \g=\g^{-k}\supset\dots \supset \g^0\supset\dots \supset \g^k,
 \end{equation*}
 where $\g^0=\p$ and $\g^1=\p_+$, and so does the isotropy representation
 \begin{equation}\label{filtr_g/p}
 \g/\p=\g^{-k}/\p\supset\dots \supset \g^{-1}/\p.
\end{equation}
Given two filtered vector spaces $\bf V$ and $\bf W$, there is an induced filtration by \emph{homogeneity} on the linear maps $\operatorname{Hom}({\bf V},{\bf W})={\bf V}^*\otimes {\bf W}$,
in which $\phi\in\operatorname{Hom}({\bf V},{\bf W})^i$ if $\phi({\bf V}^j)\subset {\bf W}^{i+j}$ for all $j$.

\subsubsection{Tractor connections}
\label{sec.tractor.connxns}
For a Cartan geometry $\bigl(M,\hat{M}, \omega\bigr)$ of any type $(\g,P)$, the natural vector bundles associated to $G$-representations play a distinguished role.
The Cartan connection $\omega$ can be canonically extended to a principal connection on the principal $G$-bundle $\tilde{M}= \hat{M} \times_PG$ and as such it induces a linear connection
$\nabla^\omega$ on all vector bundles of the form
\[\mathcal V=\tilde{M} \times_G{\bf V}= \hat{M} \times_P{\bf V}.\]
\begin{defin} Given a Cartan geometry,
the vector bundles associated to $G$-representations are called \emph{tractor bundles} and $\nabla^\omega$ the \emph{tractor connection}.
\end{defin}

The adjoint tractor bundle comes with a natural projection $\Pi\colon \mathcal A\rightarrow TM$ to the tangent bundle induced by the projection $\g \rightarrow \g/\p$,
and an algebraic action on any other tractor bundle $\mathcal A\times \mathcal V\rightarrow\mathcal V$, denoted by $(s,t)\mapsto s\bullet t$, induced by the action of~$\g$ on ${\bf V}$.

For any vector bundle $\mathcal W$ associated to any $P$-representation ${\bf W}$, there is a natural differential operator, called the \emph{fundamental derivative}, denoted by
 \begin{align*}
 D\colon\ \Gamma(\mathcal A)\times\Gamma(\mathcal W) \rightarrow \Gamma(\mathcal W),\qquad
(s,t) \mapsto D_st.
 \end{align*}
Viewing $t$ in \smash{$C^\infty\bigl(\hat{M}, {\bf W}\bigr)^P$} and $s$ in \smash{$C^\infty\bigl(\hat{M},\g\bigr)^P$}, the section \smash{$D_st \in C^\infty\bigl(\hat{M},{\bf W}\bigr)^P$} evaluates at $\hat{x} \in \hat{M}$ to the derivative of $t$ in the direction of $\omega_{\hat{x}}^{-1}(s(\hat{x}))$. The tractor connection on a tractor bundle $\mathcal V$ can then be expressed in terms of the fundamental derivative as follows:
\begin{equation}
\label{eqn.def.tractor.connxn}
\nabla_{\Pi(s)}^\omega t=D_st+s\bullet t\qquad \textrm{for}\ s\in \Gamma(\mathcal A),\, t\in\Gamma(\mathcal V).
\end{equation}
It is a routine verification that this expression defines a connection; in particular, at any point, the right-hand side only depends on $\Pi(s)$ at that point and on $t$ in a neighborhood. If $t$ is a~parallel section, it is the unique one with a given value $t(x)$ at any $x \in M$; it follows that $\iota_{\hat{x}}$ is an injective linear map to ${\bf V}$ when restricted to parallel sections in $\Gamma(\mathcal{V})$, for any $\hat{x} \in \hat{M}$.

As discussed in the introduction, parallel sections of tractor connections typically correspond to solutions of interesting equations for the geometric structure on $M$ described by the Cartan geometry. More precisely, in many contexts, certain deformations of the tractor connections by curvature terms yield invariant connections on the tractor bundles for which the parallel sections correspond to the solutions to a particular geometric PDE of interest.
 In the case of parabolic geometries, there is a powerful framework for constructing such connections, which we now briefly present.
\subsubsection{First BGG-operators of parabolic geometries}\label{sec_BGG}
Suppose $\bigl(M,\hat{M},\omega\bigr)$ is a regular normal parabolic geometry of type $(\g,P)$ with $\dim M=n$. It was shown in \cite{CD.BGG, css.bgg.seqs} that any $G$-representation ${\bf V}$ induces a \emph{BGG-sequence}
of intrinsically defined differential operators of the form
\begin{equation}\label{bgg.seq}
 \Gamma(\mathcal H_0)\stackrel{\mathcal D^{\mathcal V}_0\ }{\rightarrow}\Gamma(\mathcal H_1)\overset{\mathcal D^{\mathcal V}_1\ }{\rightarrow}\cdots
 \overset{\mathcal D^{\mathcal V}_{n-1}}{\rightarrow}\Gamma(\mathcal H_{n}),
 \end{equation}
 where \[\mathcal H_k = \mathcal H_k(M,\mathcal V)=\hat{M}\times_P H_k(\p_+,{\bf V})\] is the bundle associated to the $k$-th homology $H_k(\p_+,{\bf V})$ of $\p_+$ with coefficients in ${\bf V}$. By \cite{kostant.harmonic.curv},
 the bundles $\mathcal H_k$ are weighted tensor bundles on $M$. In the case of the homogeneous model $M \cong G/P$, the BGG-sequence is a complex, which is dual to the \emph{BGG-resolution} of ${\bf V}$ by generalized Verma modules, whence the terminology.

 The first operator $\mathcal D^{\mathcal V} =\mathcal D^{\mathcal V}_0\colon \Gamma(\mathcal H_0)\rightarrow \Gamma(\mathcal H_1)$ in \eqref{bgg.seq} always defines an overdetermined system of PDEs, the
 domain of which is easily computed to be $\mathcal H_0=\mathcal V/\mathcal V^+,$ where $\mathcal V^+ \subset \mathcal V$ is the subbundle corresponding to the $P$-submodule $\p_+{\bf V} \subset {\bf V}$.
 We denote by $\Pi^\mathcal V\colon \mathcal V\rightarrow \mathcal V/\mathcal V^+=\mathcal H_0$ the natural projection. Then, \cite{hsss.prol.connxn, neuss.prolong} shows that the projection $\Pi^\mathcal V$ restricts to a bijection between
 sections of $\mathcal V$ that are parallel for a particular linear connection on $\cV$ and solutions to~\smash{$\mathcal D^{\mathcal V}(\sigma)=0$}.

 The linear connection on $\mathcal V$, as constructed in \cite{hsss.prol.connxn}, the parallel sections of which correspond to solutions of $\mathcal D^{\mathcal V}(\sigma)=0$, is called the \emph{prolongation connection} of $\mathcal D^{\mathcal V}$ and will be denoted by~$\nabla^\mathcal V$. In general, $\nabla^\mathcal V$ differs from the tractor connection $\nabla^\omega$ by tensors built from the curvature of~$\nabla^\omega$. As above, the linear map $\iota_{\hat{x}}$ is injective when restricted to parallel sections for $\nabla^{\mathcal{V}}$ for all~${\hat{x} \in \hat{M}}$.

\subsubsection{Automorphisms acting on tractor bundles}
\label{sec.aut.tractor}

The \emph{automorphisms} of a Cartan geometry $\bigl(M,\hat{M},\omega\bigr)$ are the diffeomorphisms of $M$ that lift to principal bundle automorphisms of $\hat{M}$ which preserve $\omega$. For $f \in \operatorname{Aut}(M,\omega)$, we will also denote by $f$ its lift to $\hat{M}$, satisfying $f^* \omega = \omega$.
The automorphisms of a Cartan geometry naturally preserve the underlying geometric structure; for example, the automorphisms of the Cartan geometry associated to a conformal semi-Riemannian structure are precisely the conformal transformations.

Because it preserves the framing by $\omega$, the action of $\Aut(M,\omega)$ on $\hat{M}$ is free and proper. If $f \in \Aut(M,\omega)$ stabilizes $x \in M$, then for any $\hat{x} \in \hat{M}_x$, there is a unique $p \in P$ such that {$f.\hat{x} = \hat{x}.p$}, called the \emph{isotropy of $f$ at $x$ with respect to $\hat{x}$}.

For any $P$-module ${\bf W}$, the group $\Aut(M,\omega)$ acts on the associated bundle $\mathcal{W} = \hat{M} \times_P {\bf W}$.
The action of $f \in \Aut(M,\omega)$ on a section \smash{$X \in \Gamma(\cS)\cong C^{\infty}\bigl(\hat{M}, {\bf W}\bigr)^P$} is simply by
$f^*X = X \circ f^{-1}$.

Assume now that $\bigl(M,\hat{M},\omega\bigr)$ is a regular parabolic geometry modeled on $(\g, P)$, and that $\rho\colon G\rightarrow \textrm{GL}(\V)$ is a $G$-representation.
Let $\cV=B\times_P\V$ be the corresponding tractor bundle with tractor connection $\nabla^{\omega}$. Let $\nabla^{\cV}$ be the prolongation connection of the
first BGG-operator~$D^{\cV}$ corresponding to $\cV$.
\begin{prop} Any $f \in\emph{Aut}(M, \omega)$ preserves both $\nabla^\omega$ and $\nabla^{\cV}$, that is, $f^* \nabla^\omega = \nabla^\omega$ and~${f^*\nabla^{\cV}=\nabla^{\cV}}$.
\end{prop}

\begin{proof}
Since $f^*\omega=\omega$ for any $f \in \Aut(M,\omega)$, the fundamental derivative is invariant:
\[D_{f^* X} f^* Y = f^* (D_X Y) \qquad \forall X \in \Gamma({\mathcal{A}}), \, Y \in \Gamma(\mathcal{V}).\]
The algebraic action of $\Gamma(\mathcal{A})$ on $\Gamma(\mathcal{V})$ is evidently invariant by any principal bundle automorphism.
Thus it follows from (\ref{eqn.def.tractor.connxn}) that $f^*\nabla^\omega=\nabla^{\omega}$.

Set $\Lambda=\nabla^{\cV}-\nabla^{\omega}\in\Omega^1(M, \operatorname{End}(\cV))$ and denote by $R^{\cV}\in\Omega^2(M, \operatorname{End}(\cV))$ the curvature of~$\nabla^{\cV}$. The differential computing the Lie algebra homology of $\p_+$ with values in $\V$ induces via the identification
\[T^*M \cong \hat{M} \times_P (\g / \p)^* \cong \hat{M} \times_P \p^+\]
 a bundle map
$\partial^*_{\cV}\colon \Omega^\ell(M, \cV)\rightarrow \Omega^{\ell-1}(M, \cV)$. The connection $\nabla^{\cV}$ is characterized by the two properties \cite{hsss.prol.connxn}:
\begin{itemize}\itemsep=0pt
\item[(1)] $\Lambda\in\Omega^1(M, \operatorname{End}(\cV))^1\cap \Gamma(\operatorname{Im}(\partial^*_{\cV}\otimes {\rm Id}_{\cV^*}))$,
\item[(2)] $R^{{\cV}}\in\Gamma(\ker(\partial^*_{\cV}\otimes {\rm Id}_{\cV^*})).$
\end{itemize}
The vector bundle $(T^*M\otimes\operatorname{End}(\cV))^1$ is the subbundle of
\[
T^*M\otimes\operatorname{End}(\cV)\cong \hat M\times_P\operatorname{Hom}(\g/\p,\operatorname{End}({\bf V}))
\]
 corresponding to the first filtration component, that is, the homogeneity-one component, of the filtered $P$-representation
$\operatorname{Hom}(\g/\p,\operatorname{End}({\bf V}))$, where the filtration on that space is induced by~\eqref{filtration_rep} and~\eqref{filtr_g/p}; as it is an associated bundle to the Cartan bundle, it is acted on by $\Aut(M,\omega)$. The subbundle $\operatorname{Im}(\partial^*_{\cV}\otimes {\rm Id}_{\cV^*})$ is similarly $\Aut(M,\omega)$-invariant, so any $f \in \Aut(M,\omega)$ preserves their intersection, thus $f^*\Lambda$
is a section of this intersection. As $f^*\Lambda = f^*\nabla^\cV-\nabla^{\omega}$, the pullback~$f^* \nabla^\cV$ satisfies (1).
Similarly, the curvature of $f^*\nabla^\cV$, namely $f^*R^{{\cV}}$, is a section of the associated vector bundle $\ker(\partial^*_{\cV}\otimes {\rm Id}_{\cV^*})$, so $f^* \nabla^\cV$ satisfies (2) as well.
We conclude $f^*\nabla^{\cV}=\nabla^{\cV}$ thanks to this characterization.
\end{proof}

\subsubsection{The conformal tractor connection}
\label{sec.conf.tractor}

Let $(M,g)$ be a semi-Riemannian manifold of dimension $n \geq 3$ and let $\bigl(M, \hat{M}, \omega\bigr)$ and $\mathcal{V}$ be the corresponding parabolic Cartan geometry with model $(\g, P)$ and conformal tractor bundle, respectively, as in Section~\ref{sec.intro.tractor}. Recall that $P$ is the stabilizer of an isotropic line ${\bf V}_{1} \subset {\bf V} \cong \BR^{p+1,q+1}$. Setting ${\bf V}^{1} = {\bf V}_{1}$ and {${\bf V}^0 = {\bf V}_{1}^\perp$}, there is a $P$-invariant filtration
${{\bf V}^{1} \subset {\bf V}^0 \subset {\bf V}^{-1} = {\bf V}}$, cf.\ \eqref{filtration_rep}, giving rise to an intrinsic filtration $ \mathcal{V}^{1} \subset \mathcal{V}^0 \subset \mathcal{V}$.

A Levi complement $G_0 < P$ preserves a splitting ${\bf V} = {\bf V}_{1} \oplus {\bf V}_0 \oplus {\bf V}_{-1}$.
A choice of metric $g_0 \in [g]$ gives a corresponding decomposition
\begin{equation}
\label{eqn.conf.splitting}
\mathcal{V} = \mathcal{V}_{1} \oplus \mathcal{V}_0 \oplus \mathcal{V}_{-1}
\end{equation}
in which $\mathcal{V}_0 \cong TM$ and $\mathcal{V}_{-1}$ is identified with the line bundle of conformal scales.
References for the formulas in this section are \cite[equations~(22) and~(32)]{curry.gover.tractors}.
With respect to a decomposition~(\ref{eqn.conf.splitting}), the standard conformal tractor connection $\nabla^\omega$ is
\begin{equation}
\label{eqn.conf.connxn}
 \nabla^\omega \colon\ \begin{pmatrix} \nu \\ X \\ \sigma \end{pmatrix} \mapsto \begin{pmatrix} {\rm d} \nu - P \llcorner X \\ \nabla X + \nu \cdot {\rm Id} + \sigma P^\sharp \\ {\rm d} \sigma - X^* \end{pmatrix},
 \end{equation}
where $P^\sharp$ is the symmetric endomorphism corresponding to the Schouten tensor $P$, and all duals are with respect to $g_0$. Changing the metric $g_0 \in [g]$ corresponds to applying an element of $P \backslash G_0$ to the decomposition of tractor sections, and this transformation commutes with the tractor covariant derivative.
It follows that the intrinsic $(p+1,q+1)$-metric on $\mathcal{V}$ is $\nabla^\omega$-parallel.

Continuing with $g_0 \in [g]$, the $\sigma$-component of a $\nabla^\omega$-parallel section satisfies
\[ \nabla \mbox{grad } \sigma + \sigma P^\sharp = - \nu \cdot {\rm Id},\]
which is equivalent to (\ref{eqn.conf.einstein}) (with the sign of $\nu$ switched). Parallel sections have constant self-scalar products with respect to the tractor metric and in particular have a well-defined causal~type.

The curvature of $\nabla^\omega$ has the form, in the decomposition (\ref{eqn.conf.splitting}),
\begin{equation}
\label{eqn.conf.curv}
R^\omega(Y,Z) \colon\ \begin{pmatrix} \nu \\ X \\ \sigma \end{pmatrix} \mapsto \begin{pmatrix} - C(Y,Z)X \\ \sigma C(Y,Z)^* + W(Y,Z) X \\ 0 \end{pmatrix},
\end{equation}
where $W$ is the Weyl curvature and $C$ the Cotton tensor of $g_0$:
\[C(Y,Z)X = \frac{1}{2} \left( (\nabla_Y P)(Z,X) - (\nabla_Z P)(Y,X) \right).\]

\subsection{Borel density theorem}

The proof of the embedding theorem rests on the Borel density
theorem.
The term
\emph{algebraic} here means real-algebraic. A Zariski closure in this setting comprises the real points of the Zariski closure over ${\bf C}$.

 We use a version of the Borel density theorem which can be found in \cite[Corollary~3.2]{bfm.zimemb}, and which is due in this generality to Shalom
\cite[Theorem~3.11]{shalom.discompact} (see also \cite[Theorem~2.6]{dani.borel.density}).
The setting is a locally compact group $S$ acting on an algebraic variety $\mathbb{W}$ via a
continuous homomorphism $\rho \colon S \rightarrow \mbox{Aut } \mathbb{W}$. We
will denote by
$\bar{S}$ the Zariski closure of $\rho(S)$ in $\mbox{Aut }\mathbb{W}$.

\begin{defin}
 \label{def.discompact.rad}
 Let $G$ be an algebraic group.
The \emph{discompact radical} of $G$ is the maximal algebraic subgroup
of $G$ with no proper cocompact, algebraic, normal subgroups, denoted $G_d$.
 \end{defin}

 The discompact radical can be obtained as the minimal algebraic
 subgroup of $G$ containing all algebraic subgroups with no non-trivial compact
 algebraic quotients, which exists by the Noetherian property of
 algebraic subgroups. See \cite{shalom.discompact} for more details.

\begin{thm}[{\cite[Corollary 3.2]{bfm.zimemb}} of the Borel density theorem]
 \label{thm.borel.density}
Let $S$ be a locally compact group and $\mathbb{W}$ an algebraic variety, with
an action $\rho\colon S \rightarrow \mbox{Aut } \mathbb{W}$.
Suppose $S$ acts by homeomorphisms on a topological space $M$
preserving a finite Borel measure $\mu$. Assume $\phi\colon M \rightarrow
\mathbb{W}$ is an $S$-equivariant measurable map. Then $\phi(x)$ is fixed
by $\bar{S}_d$ for $\mu$-almost-every $x \in M$.
 \end{thm}

\section{Embedding theorem}

The following theorem is the tractor generalization of Theorem 1.2 of
\cite{bfm.zimemb} (see also Theorem~4.1 there), restated here as in the introduction with an additional technical point in the conclusions. The proof will be a straightforward adaptation of the one in~\cite{bfm.zimemb}.

\begin{thm}
\label{thm.embedding}
Let
$\bigl(M, \hat{M},\omega\bigr)$ be a Cartan geometry modeled on $(\lieg,P)$ and $B \leq \Aut(M,\omega)$ a~Lie subgroup preserving a Borel probability measure $\mu$.
Let $\rho$ be a fully reducible representation of~$G$ on $\V$ with $\rho(P)$ Zariski closed, and $\mathcal{V} = \hat{M}
 \times_\rho \V$. Let $\mathcal{S}$ be the space of parallel sections for any $\Aut(M,\omega)$-invariant connection $\nabla$ on $\mathcal{V}$.

 Then for $\mu$-almost-every $x \in M$, for all $\hat{x} \in \hat{M}_x$,
there is an algebraic subgroup $\check{B}_{\hat{x}} \leq \rho(P)$ and an algebraic epimorphism $R\colon \check{B}_{\hat{x}} \rightarrow \bar{B}_d$ such that
\[ g. \iota_{\hat{x}}(X) = \iota_{\hat{x}} {( R(g).X)} \qquad \forall X \in \mathcal{S}, \ g \in \check{B}_{\hat{x}},\]
moreover, the isotropy of $B$ at $x$ with respect to $\hat{x}$ is contained in $\check{B}_{\hat{x}}$.
\end{thm}

Recall that $\bar {B}_d$ denotes the discompact radical of the Zariski closure of
 the image of $B$ in~$\GL(\cS)$.

 \begin{proof}
Denote by $\psi\colon \Aut(M,\omega) \rightarrow \textrm{GL}(\cS)$ the representation corresponding to the restriction to $\mathcal{S}$ of the action on $\Gamma(\mathcal{V})$ as in Section~\ref{sec.aut.tractor}.
Let
\begin{align*}
\iota \colon \ \hat{M} \rightarrow \textrm{Mon}(\cS, \V), \qquad
 \hat{x} \mapsto \iota_{\hat{x}},
\end{align*}
where $ \textrm{Mon}(\cS, \V)$ is the variety of all linear injections
$\cS \rightarrow \V$. Each $g \in \Aut(M,\omega)$ acts on $\textrm{Mon}(\cS,\V)$ by pre-composition with $\psi(g)$, an algebraic automorphism, while each $p \in P$ acts by post-composition with $\rho(p)$, also an algebraic automorphism. These two actions commute.

If $g \in \Aut(M,\omega)$ stabilizes $x$ and has isotropy with respect to $\hat{x}$ equal to $p \in P$, then
\begin{equation}
\label{eqn.isotropy.in.bcheck}
\iota_{\hat{x}}(g^*X) = {p.\iota_{\hat{x}}(X)}.
\end{equation}

Let $\mathbb{U} = \mbox{Mon}(\cS, \V)/P$ and $\bar{\iota}\colon M \rightarrow \mathbb{U}$ the induced map, which is also $\Aut(M,\omega)$-equivariant.
Note that if $g \in \Aut(M,\omega)$ stabilizes $x$, then $\psi(g)$ stablizes $\bar{\iota}(x)$.
Although $\mathbb{U}$ is not generally an
algebraic variety, the Rosenlicht stratification
\cite{rosenlicht.brasil} (see also \cite[Section~2.2]{gromov.rgs}) yields a
decomposition
\[\mathbb{U} = \mathbb{U}_0 \cup \cdots \cup \mathbb{U}_r\]
in which
each $\mathbb{U}_i$ is a variety on which $\Aut(M,\omega)$ acts via $\psi$ by automorphisms
(see \cite[Theorem~4.1]{bfm.zimemb} for more details).
The sets $M_i = \bar{\iota}^{\,-1}(\mathbb{U}_i)$ are
$B$-invariant and measurable for each $i=0, \dots, r$. For each $i$ such that $\mu_i = \left. \mu \right|_{M_i}$ is not zero, it is a nontrivial finite,
$B$-invariant measure.

By Theorem~\ref{thm.borel.density}, for $\mu_i$-almost-every $x\in M_i$ the point $\bar{\iota}(x)$ is fixed by
$\bar B_d$.
Let $F_i$ comprise the $\bar{B}_d$-fixed points in
$\mathbb{U}_i$, and set $\Omega_i = (\bar{\iota} \circ \pi)^{-1}(F_i)$, a $P$-invariant subset
of $\hat{M}$ with ${\mu_i(M_i \backslash \pi(\Omega_i)) = 0}$.
For any $\hat{x} \in \Omega_i$, the orbits
\[ \iota(\hat{x}).\bar B_d \subset P.\iota(\hat{x}).\]
If we now set
\[\check{B}_{\hat{x}}=\bigl\{p\in \rho(P)\mid p.\iota(\hat{x})=\iota(\hat{x}).g \textrm{ for some } g\in\bar B_d\bigr\},\]
then $\check{B}_{\hat{x}}$ is an algebraic subgroup of $\rho(P)$ satisfying the conclusions of
the theorem for $\hat{x}$; in particular, it contains the isotropy of $B$ with respect to $\hat{x}$ by (\ref{eqn.isotropy.in.bcheck}). Let $\Omega = \cup_i \Omega_i$, a full measure set for $\mu$ comprising points that all satisfy the conclusions of the theorem.
\end{proof}

\begin{rem}
The conclusions of the theorem clearly hold when $\mathcal{S}$ is replaced by any $B$-invariant subspace of $\mathcal{S}$.
\end{rem}

For future use, we make the following general definition.

\begin{defin}
\label{def.rep.embedding}
Let $\psi\colon B \rightarrow \GL({\W})$ and $\rho \colon P \rightarrow \GL({\V})$ be two representations, with $\rho(P)$ algebraic. An \emph{embedding of $({\W}, \psi(B))$ into $({\V}, \rho(P))$} comprises a linear injection $\iota\colon {\W} \rightarrow {\V}$ together with an algebraic subgroup $\check{B} \leq \rho(P)$ and an algebraic epimorphism $R\colon \check{B} \rightarrow \bar{B}_d$ such~that
\[ g. \iota(X) = \iota( R(g).X) \qquad \forall X \in {\bf W}, \, g \in \check{B}.\]
\end{defin}

As before, $\bar{B}_d$ denotes the discompact radical of the Zariski closure of $\psi(B)$.
Note that whenever there is a Lie group monomorphism $H \rightarrow G$ for which $\Ad G$ is algebraic, then there is an embedding of $(\lieh, \Ad H)$ into $(\g, \Ad G)$.

\begin{rem}
Under the hypotheses of Theorem~\ref{thm.embedding}, let
$K \colon \hat{M} \rightarrow \mathbb{W}$ be an $\Aut(M,\omega)$-invariant, $P$-equivariant map to a variety $\mathbb{W}$ on
 which $P$ acts algebraically.
 In the proof of the theorem, we can take
\[\mathbb{U} =\textrm{Mon}(\cS, \V)\times \mathbb{W}\]
with $\Aut(M,\omega)$ acting by precomposition via $\psi$ on the first factor and trivially on the second factor, commuting with the $P$-action on the product.
Considering the map $\iota \times K \colon \hat{M}\rightarrow \mathbb{U}$ given
by $\hat{x} \mapsto (\iota_{\hat{x}}, K(\hat{x}))$, the same proof with the Borel density theorem as above yields, for almost-every $x \in M$, not only the embedding of $(\mathcal{S}, \psi(B))$ into $({\bf V}, \rho(P))$ but also that $K(\hat{x})$ is fixed by $\check{B}_{\hat{x}}$; the conclusion that $\check{B}_{\hat{x}}$ contains the isotropy of $B$ at $x$ with respect to $\hat{x}$ moreover remains true.
\end{rem}

A useful choice of $\mathbb{W}$ in the remark above is, for $m \in {\bf N}$, the $P$-module in which the total
order-$m$ fundamental derivative of the Cartan curvature
\[D^m \kappa\colon \ \hat{x} \mapsto \bigl( \kappa(\hat{x}), D \kappa(\hat{x}), \dots, D^{(m)} \kappa(\hat{x}) \bigr)\]
takes its
values (see Sections~\ref{sec.tractor.connxns} and \ref{sec.aut.tractor}). Fixed points of $P$ in this module can give rise to local automorphisms
by the Frobenius theorem for Cartan geometries. More precisely, for~$\bigl(M,\hat{M}, \omega\bigr)$ real-analytic and $M$ closed, there is $m \geq 0$ such that for any $\hat{x} \in \hat{M}$, any $p \in P$ fixing $D^m \kappa(\hat{x})$ is the isotropy at $x = \pi(\hat{x})$ with respect to $\hat{x}$ of a unique $g \in \Aut^{\rm loc}(M,\omega)$ fixing~$x$ (where uniqueness applies to any two such local automorphisms defined in the same neighborhood of~$x$) \cite{me.frobenius}, see also \cite{pecastaing.frobenius} for a smooth version. The following corollary of the embedding theorem is well known in the setting of infinitesimal automorphisms (see \cite[Theorem~5.3]{me.frobenius}).

\begin{cor}
\label{cor.emb.frob}
Let
$\bigl(M, \hat{M},\omega\bigr)$ be a real-analytic Cartan geometry modeled on $(\lieg,P)$ and $B \leq \Aut(M,\omega)$ a Lie subgroup preserving a Borel probability measure $\mu$.
Let $\rho$ be a fully reducible, real-analytic representation of $G$ on $\V$ with $\rho(P)$ Zariski closed, and $\mathcal{V} = \hat{M}
 \times_\rho \V$. Let $\mathcal{S}$ be the space of parallel sections for any $\Aut(M,\omega)$-invariant connection $\nabla$ on $\mathcal{V}$.

 Then for $\mu$-almost-every $x \in M$, for all $\hat{x} \in \hat{M}_x$, in addition to the conclusions of Theorem~$\ref{thm.embedding}$, every $g \in \check{B}_{\hat{x}}$ is the isotropy with respect to $\hat{x}$ of an element of $\Aut^{{\rm loc}}(M,\omega)$ fixing~$x$.
\end{cor}

Parallel tractor sections, like infinitesimal automorphisms, correspond to solutions of overdetermined PDEs, and a Frobenius theorem in terms of fundamental derivatives of the tractor curvature is possible. We have worked out such a result, however, we have not seen how to profitably combine it with the output of the embedding theorem, so it is not included in the present paper.

\begin{rem}
\label{rem.conjugate.embedding}
Let $\hat{x}$ be as in the conclusion of Corollary \ref{cor.emb.frob}. The statement says the conclusions hold for $\hat{x}.g$ for all $g \in P$, as well. In this case, the isotropy of $B$ with respect to $\hat{x}.g$ is that with respect to $\hat{x}$, conjugated by $g^{-1}$. We also have
\[D^m \kappa(\hat{x}.g) = g^{-1}. D^m \kappa(\hat{x}) \qquad \mbox{and} \qquad \check{B}_{\hat{x}.g} = g^{-1} \check{B}_{\hat{x}} g.\]
The group $\check{B}_{\hat{x}}$ will thus be conjugated in $P$ as needed below, corresponding to choosing a~different~${\hat{x} \in \hat{M}_x}$, without altering the properties ensured by the embedding theorem.
\end{rem}

\section{Application in conformal geometry: Algebraic part}

The remainder of the paper is devoted to the proof of Theorem~\ref{thm.conf.application}, which is divided into three parts. The first two parts are Lie-algebraic, while the shorter, third part has the application of our embedding theorem to the conformal tractor bundle and the conclusion of conformal flatness. In the first part, we prove the upper bound, namely, the following proposition. Recall the notion of embedding between representations from Definition~\ref{def.rep.embedding}.

\begin{prop}
\label{prop.sutoso.ineq}
Let $H$ be locally isomorphic to $\SU(p',q')$ for $0 < p' < q'$, with Borel subgroup~$B_H$. Let $\lieg = \so(p+1,q+1)$ with $p \leq q$, and $P < G$ the stabilizer of an isotropic line of $\BR^{p+1,q+1}$. Suppose there is an embedding of~$(\lieh, {\rm Ad}_\h B_H)$ in $(\lieg, {\rm Ad}_\g P)$ with $\operatorname{ad}(\iota(\lieh) \cap \p ) \subseteq \check{\b}$. Then $p \geq 2 p'-1$.
\end{prop}

In Section~\ref{subsec.alignment}, we ensure that the maximal $\BR$-split tori of ${\rm Ad}_{\h} B_H$ and ${\rm Ad}_\g P$ are aligned under the projection $R$ given by Theorem~\ref{thm.embedding}. In Section~\ref{subsec.consequences.stab.iso}, we analyze $\iota\colon \h \rightarrow \g$ by studying its restriction to $\h^0 = \iota^{-1}(\p)$, which is a Lie algebra homomorphism (see~(\ref{eqn.iota.in.p})), and the induced $\bar{\iota}$ from $\h/\h^0$ to $\g/\p$, which are compared as representations via the epimorphism $R$. Key conclusions, recorded in Lemma~\ref{lem.h.mod.h0}, are bounds in terms of $p'$ on maximal isotropic subspaces in $\h/\h_0$, for a class of scalar products induced by the embedding. Based on these, we prove Proposition~\ref{prop.sutoso.ineq} in Section~\ref{subsec.proof.ineq}.

The alignment of $\BR$-split tori from Section~\ref{subsec.alignment} gives rise to a homomorphism $\theta$ between the duals of the their Lie algebras, restricting to a function from the roots of $\g$ to the roots of $\h$. When $p=2p'-1$, then $\theta$ is essentially determined and is 2-to-1 between these sets of roots. The features of $\theta$ are recorded in Proposition~\ref{prop.sutoso.roots}, which is proved in Section~\ref{subsec.root.correspondences}, completing the algebraic part of the proof of Theorem~\ref{thm.conf.application}.

\subsection{Preliminaries: Root systems and notation}

We set $G = {\rm O}(p+1,q+1)$ and $H$ a group locally isomorphic to ${\rm SU}(p',q')$, with $p\leq q$ and~${p' < q'}$, as in Theorem~\ref{thm.conf.application} and Proposition~\ref{prop.sutoso.ineq}. It is always assumed $p, p' > 0$.

\begin{rem}
If $p<q$, then the root system of $\lieo(p+1,q+1)$ is of type $B_{p+1}$, while if $p=q$, then it is type $D_{p+1}$. If $p'<q'$, then the root system of $\mathfrak{su}(p',q')$ is of type $BC_{p'}$, while if $p'=q'$, then it is of type $C_{p'}$. The proof of Proposition~\ref{prop.sutoso.roots} involves careful analysis of how the roots of~$BC_{p'}$ can correspond with those of~$B_{p+1}$ or $D_{p+1}$ under the embedding. It is not valid for~$C_{p'}$, thus we always assume $p' < q'$.
\end{rem}

If $p<q$, let $\beta_1, \dots, \beta_{p+1} \in \a^*$ be such that the positive roots of $\lieg$ are
\[ \Delta^+_G = \{ \beta_k \} \cup \{ \beta_i - \beta_j \mid i < j \} \cup \{ \beta_i + \beta_j \mid i \neq j\}.\]
The root system is $B_{p+1}$.
The set $\{ \beta_k \}$ are called the positive short roots, and their root spaces are each of dimension $q-p$. The remaining two subsets are the long roots, such that the $\{ \beta_i - \beta_j \}$ are the \emph{initial roots} of a string of positive roots, while the $\{ \beta_i + \beta_j \}$ are the \emph{terminal roots} of a string of positive roots. Either of these types of roots have $1$-dimensional root spaces. These strings are such that for any $X \in \g_{\beta_i - \beta_j}$ and $Y \in \g_{\beta_j}$ both nonzero, the brackets $[X,Y] \in \g_{\beta_i}$ and $[[X,Y],Y] \in \g_{\beta_i + \beta_j}$ are nonzero. Also for nonzero $X \in \g_{\beta_i - \beta_j}$ and~$Y \in \g_{\beta_j + \beta_k}$, the bracket~$[X,Y]$ spans $\g_{\beta_i + \beta_k}$.
 The weights of the $P$-representation on $\lieg/ \p$ are
\[\{ \beta_i - \beta_1, - \beta_1, - \beta_1 - \beta_i \mid i =2, \dots, p+1 \},\]
in particular, the conformal factor is given by $- 2 \beta_1$.

If $p=q$, we choose similarly $\beta_1, \dots, \beta_{p+1} \in \a^*$ such that the positive roots of $\lieg$ are
\[ \Delta^+_G = \{ \beta_i - \beta_j \mid i < j \} \cup \{ \beta_i + \beta_j \mid i \neq j\}.\]
The root system is $D_{p+1}$. All these positive roots are of the same length and have $1$-dimensional root spaces. The weights of the representation $\lieg/ \p$ are the same as in the case $p<q$, except that $-\beta_1$ is not a weight; the conformal factor is still $- 2 \beta_1$.

We choose a set of short roots $\alpha_1, \dots, \alpha_{p'} \in \Delta_H$ such that the positive roots of $\lieh$ are
\[ \Delta_H^+ = \{ \alpha_{k'} \} \cup \{ \alpha_{i'} - \alpha_{j'} \mid i' < j' \} \cup \{ \alpha_{i'} + \alpha_{j'} \} .\]
This is the root system $BC_{p'}$. The root spaces for the short roots $\alpha_{k'}$ are of dimension $2(q'-p')$. The root spaces for $\alpha_{i'} - \alpha_{j'}$ or $\alpha_{i'} + \alpha_{j'}$ with $i' \neq j'$ are $2$-dimensional, while those for $2 \alpha_{i'}$ are $1$-dimensional.

For the first application of the embedding theorem, we consider a Borel subgroup $B_H < H$. Note that $B_H$ is solvable, algebraic, and discompact. As $B_H$ is amenable and $M$ is compact, there is a $B_H$-invariant probability measure $\mu$ on $M$; see \cite[Definition~4.1.1 and Corollary~4.1.7]{zimmer.etsg}.
Let $\check{B} = \check{B}_{\hat{x}} < {\rm Ad}_\lieg P$ be the algebraic subgroup given by Theorem~\ref{thm.embedding} applied to $B_H$ acting on the adjoint tractor bundle endowed with \v{C}ap's connection as in Section~\ref{sec.intro.tractor}. The parallel sections correspond to infinitesimal automorphisms, and we restrict to $\mathcal{S} = \lieh$. Note that this version of the embedding theorem is exactly the one of \cite{bfm.zimemb}.

\subsection{Alignment of the intertwining epimorphism}
\label{subsec.alignment}

Recall that the parabolic subgroup $P < \SO(p+1,q+1)$ is the stabilizer of an isotropic line in~$\BR^{p+1,q+1}$. A~\emph{null-orthonormal basis} of $\BR^{p+1,q+1}$ is one with respect to which the quadratic form has the form, writing $n=p+q$,
\[q(x_1, \dots, x_{n+2}) = 2 \sum_{i=1}^{p+1} x_i x_{n+3-i} + \sum_{i=p+2}^{q} x_i^2.\]

The Lie algebra $\check{\b}$ is a subalgebra of $\p$, with adjoint action preserving the image $\iota(\lieh)$, such that the restriction to this image is a surjection onto $\iota \circ \operatorname{ad}_\lieh \b_H \circ \iota^{-1} \cong \b_H$.
We will denote by~$r$ the resulting Lie algebra epimorphism $\check{\b} \rightarrow \b_H$.
We can focus on the solvable radical $\check{\b}^{\rm solv}$, as any semisimple Levi factor is in the kernel of $r$. Because it is algebraic, $\check{\b}^{\rm solv}$ is invariant under Jordan decomposition (see \cite[Theorem~4.3.3]{morris.ratners.thms.book}). Since ${\rm Ad}_\h B_H$ has no elliptic elements, any elliptic elements of $\check{\b}^{\rm solv}$ are in the kernel of $r$; more precisely, they belong to a compact torus of $\ker r$, which is a direct factor of $\check{\b}^{\rm solv}$ (see \cite[Theorem~4.4.7 and Corollary 4.4.12]{morris.ratners.thms.book}), so we will simply exclude this factor and assume that all elements of $\check{\b}^{\rm solv}$ have trivial elliptic components in their Jordan decomposition.

\begin{lem}
\label{lem.solvable.in.p}
Let a null-orthonormal basis of $\BR^{p+1,q+1}$ in which $P$ is block-upper-triangular
be given. A solvable, algebraic
subgroup $\check{B}^{\rm solv} < P$ containing no elliptic elements is conjugate in~$P$ to a group of upper triangular matrices.
A maximal $\BR$-split torus $A_G < G$ can be chosen to be the maximal diagonal subgroup of $P$ in the resulting null-orthonormal basis.
\end{lem}

\begin{proof}
Recall that $P$ is the stabilizer of a certain isotropic line, which we will denote $\ell_0$. The quotient $\ell_0^\perp/\ell_0 = V^{p,q}$ inherits a conformal class of
scalar products of signature $(p,q)$. On $V^{p,q}$ the solvable group $\check{B}^{\rm solv}$ descends to preserve this conformal class and $P$ descends to the full linear conformal group ${\rm CO}(p,q)$. The claim of the lemma is equivalent to the assertion that
a~solvable, algebraic subgroup of ${\rm CO}(p,q)$ without elliptic elements is conjugate in ${\rm CO}(p,q)$ to a~group of upper-triangular matrices.

Lie's theorem ensures that in $V^{p,q}$ there is a simultaneous eigenvector $v_1$ or a $2$-dimensional irreducible subspace for $\check{B}^{\rm solv}$. In the latter case, the restriction of $\check{B}^{\rm solv}$ splits into a nontrivial elliptic and a scalar action. By the assumption that all elements of $\check{B}^{\rm solv}$ have trivial elliptic Jordan components, this possibility is excluded. If the eigenvector $v_1$ is timelike or spacelike, then the eigenvalues on it are $\pm 1$. In this case,
apply Lie's theorem again to $v_1^\perp$, to conclude that it contains an isotropic simultaneous eigenvector, or $\check{\b}^{\rm solv}$ is trivial on $V^{p,q}$. Proceeding inductively yields a null-orthonormal basis in which $\check{B}^{\rm solv}$ is upper-triangular. It follows that on $V^{p,q}$ the conjugacy of $\check{B}^{\rm solv}$ to upper-triangular matrices can be achieved in ${\rm CO}(p,q)$. Consequently, on~$\BR^{p+1,q+1}$, the conjugacy can be achieved in $P$, since $\check{B}^{\rm solv}$ preserves $\ell_0$, and we can take the last basis vector to be a null vector not orthogonal to $\ell_0$.

In any null-orthonormal basis in which $P$ is block-upper-triangular -- that is, in which $P$ is the stabilizer of the line spanned by the first element of the basis -- the maximal diagonal subgroup is a maximal $\BR$-split torus of $G$. In particular, such a diagonal subgroup in the basis we have found here can be taken to be $A_G$.
\end{proof}

As explained in Remark \ref{rem.conjugate.embedding}, the group $\check{B}^{\rm solv}$ can be replaced with the $P$-conjugate given by the above lemma, and all conclusions of the embedding theorem continue to hold.

 \begin{lem}
After conjugating $\check{B}^{\rm solv}$ by a suitable element of $P$, there is a subalgebra $ \check{\a} < \check{\b}^{\rm solv} \cap \a_G$ mapping isomorphically under $r$ onto $\a_H$, and this conjugation preserves the properties of~$\check{B}^{\rm solv}$ achieved in Lemma~$\ref{lem.solvable.in.p}$.
 \end{lem}

 \begin{proof}
 Fix a basis $X_1, \dots, X_{p'}$ of $\a_H$. Let $\check{X}_1$ be a lift of $X_1$ to $\check{\b}^{\rm solv}$. The semisimple component of $\check{X}_1$ in its Jordan decomposition in $\g$ is nontrivial, so it is contained in $\check{\b}^{\rm solv}$ by algebraicity of the latter. Thus write $\check{X}_1 = A_1 + Z_1$ where $A_1$ is $\BR$-semisimple in $\g$ and $Z_1 \in \ker r$. The subalgebra $r^{-1}(\a_H)$ admits an $\ad A_1$-invariant subspace decomposition $\check{\a}_1 \ltimes \ker r$. Now $\check{\a}_1$ maps bijectively under $r$ to $\a_H$, which is abelian, so it commutes with $A_1$.

 Let $A_2$ be a lift of $X_2$ to $\check{\a}_1$. The semisimple Jordan component of $A_2$ commutes with $A_1$, has the same projection to $\a_H$ as $A_2$, and is contained in $\check{\b}^{\rm solv}$ by algebraicity. So we can assume~$A_2$ is semisimple. The $0$-eigenspace of $\ad A_1$, denote it $E_0(A_1)$, is $\ad A_2$-invariant. There is thus an~$\ad A_2$-invariant complement to $\ker r$ in $r^{-1}(\a_H)$, call it $\check{\a}_2$, which is contained in $E_0(A_1)$ and maps under $r$ onto $\a_H$. This $\check{\a}_2$ centralizes $A_1$ and $A_2$. Continuing for $i' = 3, \dots, p'$ results in an abelian subalgebra $\check{\a}_{p'} \subset \check{\b}^{\rm solv}$ mapping isomorphically under $r$ to $\a_H$ and comprising $\BR$-semisimple elements of $\g$. Such a subalgebra is contained in a maximal abelian $\BR$-split subalgebra of $\p$, which is in turn conjugate in $P$ to the desired $\check{\a}$ contained in our chosen diagonal subalgebra~$\a_G$.

The subspaces $\check{\a}_{i'}$ comprise upper-triangular matrices in $\p$, and the conjugation to the diagonal takes place in the subalgebra of $\p$ comprising upper-triangular matrices.
 \end{proof}

 Replacing $\check{B}^{\rm solv}$ with the conjugate given by the above lemma, we will henceforth have $\check{\a} \subset \a_G$ lifting $\a_H$. Denote by $\theta$ the homomorphism from $\a_G^*$ to $\a_H^*$ induced by restriction to $\check{\a}$ composed with $(r^*)^{-1}\colon \check{\a}^* \rightarrow \a_H^*$.

Finally, consider the unipotent radicals.
Because $r$ is an algebraic epimorphism, it maps the nilpotent radical of $\check{\b}^{\rm solv}$, which is contained in $\p_+$, onto that of $\b_H$, which is the sum of positive root spaces of $\lieh$ (see \cite[Corollary 4.3.6]{morris.ratners.thms.book}). The restriction of $r$ to the nilpotent radical may not split,
but there is a subspace $\check{U}_{\b}$,
complementary to $\ker r$ in the nilpotent radical of $\check{\b}^{\rm solv}$ and invariant by the $\BR$-diagonal subgroup of $\check{B}^{\rm solv}$, in particular by $\check{A}$. For $\lieh_\alpha \subset \b_H$ a positive root space, we will denote by $\check{\lieh}_\alpha$ its subspace lift to $\check{U}_\b$.

\subsection{Consequences of equivariance for stabilizer and isotropy}
\label{subsec.consequences.stab.iso}

The starting point for the analysis in this section are ideas which can be found in \cite{pecastaing.rank1}.
The facts that $\p$ is a subalgebra of $\lieg$ and that $\ad \p$ preserves a $(p,q)$-scalar product on $\lieg/\p$ up to scale have the following consequences via the embedding.

\begin{prop}[{\cite[Proposition~2.1]{pecastaing.rank1}}]\label{prop.vincent}
For $\lieh^0 = \iota^{-1}(\p)$,
\begin{itemize}\itemsep=0pt
\item[$(1)$] $\lieh^0$ is $\ad \b_H$-invariant,
\item[$(2)$] $\ad \b_H$ preserves a symmetric bilinear form on $\lieh/\lieh^0$ up to scaling by $- 2 \theta(\beta_1)$.
\end{itemize}
\end{prop}

A straightforward consequence of the second axiom for the Cartan connection $\omega$ is that for~${X \in \lieh^0}$,
\begin{equation}
\label{eqn.iota.in.p}
 \iota ( [X,Y] ) = [\iota(Y),\iota(X)] \qquad \forall Y \in \lieh.
 \end{equation}
 It immediately follows that $\lieh^0$ is a subalgebra. In the Levi decomposition $\p \cong \q \ltimes (\BR^{p,q})^*$, the Levi complement $\q \cong \BR \oplus \so(p,q)$ is precisely the sum of root spaces $\lieg_\beta$ with $\beta$ in the span of~$\{ \beta_2, \dots, \beta_{p+1} \}$.

\begin{lem}\label{lem.lots.in.h0}
Let $\lieh^0 = \iota^{-1}(\p)$.
\begin{itemize}\itemsep=0pt
\item[$(1)$] If $\lieh = \lieh^0$, then $\iota(\lieh)$ is conjugate in $P$ to a subalgebra of $\so(p,q) \cong \q_{ss} \subset \q$.
\item[$(2)$] If $\lieh_\alpha \subset \lieh^0$ for $\alpha \in \Delta^+_H$, then the choice of the subspace $\check{U}_\b$ can be modified so that ${\check{\lieh}_\alpha = \iota(\lieh_\alpha) \subset \check{\b}}$.
\end{itemize}
\end{lem}

\begin{proof}
If $\lieh=\lieh^0$, then (\ref{eqn.iota.in.p}) and injectivity of $\iota$ imply the image of $\iota$ is a semisimple subalgebra of $\p$. Now $\iota(\lieh) = \iota_{\hat{x}}(\lieh) \subseteq \p$ for $\hat{x}$ given by the embedding theorem, and it is an isotropy algebra at the point $x = \pi(\hat{x})$.
The Thurston stability theorem \cite[Theorem~3]{thurston.reeb.stability} implies that $\lieh$ is linearizable at $x$, which means $\iota(\lieh)$ is conjugate by an element $g \in P_+$ into $\q$.
Semisimplicity implies the image is in $\q_{ss}$.

 Under the weaker assumption of (2), the conclusions of the embedding theorem still give that~$\iota(\lieh_\alpha) \subseteq \check{\b}$. It moreover belongs to the solvable subalgebra $\iota(\b_H) \cap \p \subset \check{\b}$, which is an~ideal by~(\ref{eqn.iota.in.p}). Therefore, $\iota(\lieh_\alpha) \subset \check{\b}^{\rm solv}$. As $\iota(\lieh_\alpha)$ comprises nilpotent elements, it belongs to the nilpotent radical of $\check{\b}^{\rm solv}$. Finally, it is $\check{\a}$-invariant. Thus $\iota(\lieh_\alpha)$ could differ from the original choice of $\check{\lieh}_\alpha$ only by a linear function to the intersection of the nilradical of $\check{\b}^{\rm solv}$ with $\ker r$, and we are free to change our choice of $\check{U}_{\b}$ accordingly.
\end{proof}

\begin{lem}[useful facts about $\lieh^0$]
\label{lem.about.h0}
As above, let $\lieh^0 = \iota^{-1}(\p)$.
\begin{itemize}\itemsep=0pt
\item[$(1)$] If $\a_H \subset \lieh^0$, then $\b_H \subset \lieh^0$.
\item[$(2)$] For $i'$, $j'$, possibly equal, if $\lieh_{- \alpha_{i'} - \alpha_{j'}} \cap \lieh^0 \neq 0$, then $\lieh_{- \alpha_{i'}} \subset \lieh^0$.
\item[$(3)$] Suppose $p' \geq 2$ and let $i' \neq j'$. If $\lieh_{-\alpha_{i'}} \subset \lieh^0$ and $\lieh_{- \alpha_{j'}} \cap \lieh^0 \neq 0$, then $\lieh_{-\alpha_{j'}} \subset \lieh^0$.
\end{itemize}
\end{lem}

\begin{proof}
These facts all follow easily from the $\ad \b_H$-invariance of $\lieh^0$ in Proposition~\ref{prop.vincent}\,(1), together with bracket relations in $\lieh$:
\begin{itemize}\itemsep=0pt
\item[(1)] $\b_H = [\b_H, \a_H] + \a_H$.
\item[(2)] For any nonzero $X \in \lieh_{- \alpha_{i'} - \alpha_{j'}}$, the image $[\lieh_{\alpha_{j'}},X] = \lieh_{-\alpha_{i'}}$.
\item[(3)] For any nonzero $X \in \lieh_{-\alpha_{j'}}$, there is $Y \in \lieh_{- \alpha_{i'}}$ with $0 \neq [X,Y] \in \lieh_{- \alpha_{i'}-\alpha_{j'}}$. From the fact that $\lieh^0$ is a subalgebra and (2), the conclusion follows.
\hfill\qed
\end{itemize}
\renewcommand{\qed}{}
\end{proof}

For a subspace $\mathfrak{l}$ or element $X$ of $\lieh$, we will denote by $\bar{\mathfrak{l}}$ or $\bar{X}$, respectively, the projection to~$\lieh/\lieh^0$.
The next lemma will be key for establishing the bound between $p'$ and $p$. The general idea is well known, and can be found in particular in \cite{pecastaing.rank1}.

\begin{lem}[useful facts about $\lieh/\lieh^0$]
\label{lem.h.mod.h0}
Suppose that for some $i'$, $\bar{\lieh}_{- \alpha_{i'}} \neq 0$. Denote by $[s]$ the class of the bilinear form on $\lieh/\lieh^0$ from Proposition~$\ref{prop.vincent}$\,$(2)$.
\begin{enumerate}\itemsep=0pt
\item[$(1)$] If $p' >1$ and $\bar{\lieh}_{- \alpha_{j'}} \neq 0$ for $j' \neq i'$, then $[s]$ has an isotropic subspace of dimension greater than $2p' - 1$, unless $p'=2$, $\theta(\beta_1) = \alpha_1 + \alpha_2$, and $[s]$ has an isotropic subspace of dimension~$3$.
\item[$(2)$] If $p' > 1$ and $\lieh_{- \alpha_{j'}} \subset \lieh^0$ for all $j' \neq i'$, then $[s]$ has an isotropic subspace of dimension at least $2p'-1$. If the maximal isotropic subspace for $[s]$ is dimension $2 p' - 1$, then
 \[\theta(2 \beta_1) = 2 \alpha_{i'} + \alpha'\qquad \mbox{for} \ \alpha' \in \{ 0, \alpha_{j'}, \alpha_{j'} + \alpha_{k'} \mid j', k' \neq i' \}.\]
\end{enumerate}
\end{lem}

\begin{proof}
As in \cite[Observation 3]{pecastaing.rank1}, the idea is that if $\alpha + \alpha'$ does not equal the conformal factor~$- \theta(2 \beta_1)$, then $\overline{\lieh}_\alpha \perp \overline{\lieh}_{\alpha'}$ for $[s]$.

Under the hypotheses of (1), Lemma~\ref{lem.about.h0}\,(2) implies that ${\lieh}_{- \alpha_{i'} - \alpha_{k'}}$ and ${\lieh}_{- \alpha_{j'} - \alpha_{k'}}$ have trivial intersection with $\lieh^0$ for all $k'$. The roots
\[
\Sigma = \{ - \alpha_{i'} - \alpha_{k'}, - \alpha_{j'} - \alpha_{k'}, - \alpha_{i'}, - \alpha_{j'} \mid 1 \leq k' \leq p' \}
\]
all occur as weights on $\lieh/\lieh^0$. Any sum of two of them has the form \smash{$\sum_{k'=1}^{p'} - n_{k'} \alpha_{k'}$} where $n_{k'} \in \N$, \smash{$2 \leq \sum n_{k'} \leq 4$}, and $n_{i'} + n_{j'} \geq 2$.
Each such sum is obtained in one or two ways -- that is, it is the sum of one or two different unordered pairs of elements of $\Sigma$. Excluding at most two elements of $\Sigma$ yields a subset $\Sigma'$ so that the sum of any two elements of $\Sigma'$ does not~equal~$- \theta(2 \beta_1)$. The root spaces \smash{$\lieh_{- \alpha_{k'} - \alpha_{\ell'}}$} with $k' \neq \ell'$ have dimension 2, while \smash{$\lieh_{- 2 \alpha_{i'}}$}, \smash{$\lieh_{- 2 \alpha_{j'}}$} are one-dimensional; the spaces \smash{$\overline{\lieh}_{- \alpha_{i'}}$}, \smash{$\overline{\lieh}_{- \alpha_{j'}}$} have dimension at least $1$. Thus \smash{$\bigoplus_{\sigma \in \Sigma'} \overline{\lieh}_\sigma$} is totally isotropic for $[s]$ and has dimension at least $4p'-6$. If $p' \geq 3$, then the latter dimension exceeds $2p'-1$.

Now suppose $p'=2$, so
\[\Sigma = \{ -\alpha_1 - \alpha_2, - 2 \alpha_1, - 2 \alpha_2, - \alpha_1, - \alpha_2 \}.\]
For a sum $\sigma = - n_1 \alpha_1 - n_2 \alpha_2$ with $n_1 \leq 1$, excluding $- \alpha_2$ and $- 2 \alpha_2$ gives $\Sigma'$ with no two elements summing to $\sigma$ and corresponding to a subspace of $\lieh/\lieh^0$ of dimension at least 4. Thus we may assume $ \theta(2 \beta_1) \neq n_1 \alpha_1 + n_2 \alpha_2$ with $n_1 \leq 1$, nor, by symmetry, with $n_2 \leq 1$. The only possibility is thus that $\theta(2 \beta_1) = 2 \alpha_1 + 2 \alpha_2$, with an isotropic subspace of dimension 3.

In the case (2), the root spaces $\lieh_{- \alpha_{i'} - \alpha_{j'}}$ have trivial intersection with $\lieh^0$ for all $j'$, and Lemma~\ref{lem.about.h0}\,(3) implies that $\lieh_{- \alpha_{i'}} \cap \lieh^0 = 0$ as well. For the collection of roots $ \Sigma = \{ {- \alpha_{i'} - \alpha_{k'}},\allowbreak - \alpha_{i'} \mid 1 \leq k' \leq p' \}$, any sum of two elements is obtained from a unique unordered pair, and has the form $- 2 \alpha_{i'} - \alpha'$ for $\alpha'$ as in the conclusion of (2), or $\alpha' = \alpha_{i'}$ or $2 \alpha_{i'}$. If $- \theta(2 \beta_1) $ is not such a sum, then all of the root spaces $\lieh_{- \alpha_{i'} - \alpha_{k'}}$ together with $\lieh_{- \alpha_{i'}}$ comprise an isotropic subspace, and it has dimension $2 p' +1$. If $\theta(2 \beta_1) = 2 \alpha_{i'} + \alpha'$ with $\alpha' = \alpha_{i'}$ or $2 \alpha_{i'}$, then $\lieh_{- 2 \alpha_{i'}}$ can be omitted, leaving an isotropic subspace of dimension $2p'$. Thus the maximal isotropic subspace can have dimension $2p'-1$ only if $\alpha'$ is as in the conclusion of (2).
\end{proof}

\subsection{Proof of the inequality}\label{subsec.proof.ineq}

This section is devoted to the proof of Proposition~\ref{prop.sutoso.ineq}. Any $[s]$-isotropic subspace of $\lieh/\lieh^0$ has dimension at most $p$.
As in~\cite{pecastaing.rank1},
the tension between this upper bound and the isotropic subspaces given by Lemma~\ref{lem.h.mod.h0} plays a key role.

If $p=1$, then $p < q$ because $\mbox{dim } M \geq 3$.
The bound $p'=1$ then follows from Zimmer~\cite{zimmer.rankbounds} and Bader--Nevo \cite[Theorem~2]{bn.simpleconf}; see also Pecastaing \cite{pecastaing.smooth.sl2r}.

\begin{prop}
If $p=2$, then $p'=1$.
\end{prop}

\begin{proof}
First, suppose $\lieh_{- \alpha_{i'}}$ is contained in $\lieh^0$ for all $i'$. The $\ad \b_H$-invariance of $\lieh^0$ then implies that all nonnegative root spaces are in $\lieh^0$, as well. Now $\b_H$ and all $\lieh_{\pm \alpha_{i'}}$ are in $\lieh^0$. These subspaces generate $\lieh$, so $\lieh = \lieh^0$ by (\ref{eqn.iota.in.p}). Lemma~\ref{lem.lots.in.h0}\,(1) implies that by conjugating in $P$, we may assume $\iota(\lieh) \subset \q_{ss} \cong \so(2,q)$, see Remark \ref{rem.conjugate.embedding}. Provided $q \geq 3$, the conclusions of the case~${p=1}$ apply -- note that now $\lieh$ acts conformally on the Lorentzian M\"obius space of dimension~$q$. On the other hand, if $p=q=2$, then there is no nontrivial Lie algebra homomorphism from $\mathfrak{su}(p',q')$ to $\so(2,2) \cong \so(1,2) \oplus \so(1,2)$ where $0 < p' < q'$.
Thus $p'=1$.

Next suppose that one $\overline{\lieh}_{- \alpha_{i'}} \neq 0$ but that $\lieh_{- \alpha_{j'}} \subset \lieh^0$ for all $j' \neq i'$. An isotropic subspace for $[s]$ has dimension at most $2$; on the other hand, Lemma~\ref{lem.h.mod.h0}\,(2) gives the lower bound $2p'-1$ if $p' >1$, so $p'=1$.

Finally, suppose that $\overline{\lieh}_{- \alpha_{i'}}$, $\overline{\lieh}_{- \alpha_{j'}} \neq 0$ for distinct $i'$, $j'$. This could only happen if $p'>1$, in which case Lemma~\ref{lem.h.mod.h0}\,(1) would give an $[s]$-isotropic subspace of dimension at least $3$, so this case does not occur.
\end{proof}

Now let $p \geq 3$.
If for some short root $\alpha_{i'}$, the projection $\overline{\lieh}_{- \alpha_{i'}} \neq 0$, then by Lemma~\ref{lem.h.mod.h0}, $p \geq 2p'-1$.

If all $\lieh_{- \alpha_{i'}} \subset \lieh^0$, then by Lemma~\ref{lem.lots.in.h0}\,(1) and (\ref{eqn.iota.in.p}), $\iota$ can be conjugated in $P$ to a homomorphism into $\q_{ss} \cong \so(p,q)$. Then a group isogeneous to $H$ acts conformally on $\mbox{\bf M\"ob}^{p-1,q-1}$.
Thus there is an embedding of $(\lieh, \Ad B_H)$ into $(\q, {\rm Ad}_\q P')$, where $\q = \q_{ss}$ and $P'$ is the stabilizer of an isotropic line in $\BR^{p,q}$ (by Theorem~\ref{thm.embedding} and amenability of $B_H$).
By induction on $p$, we~conclude~${p' \leq p/2 < (p+1)/2}$.

\subsection{Root correspondences under the embedding}
\label{subsec.root.correspondences}

For $\alpha \in \Delta_H^+$, denote by
\[\check{\Delta}(\alpha) = \bigl\{ \beta \in \theta^{-1}(\alpha) \cap \Delta_G^+ \mid \exists \check{X} \in \check{\lieh}_\alpha \ \mbox{with} \ \check{X}_\beta \neq 0 \bigr\},\]
where $\check{X}_\beta$ denotes the component on the subspace $\lieg_\beta$.
 Let
 \[\check{\Delta} = \bigcup_{\alpha \in \Delta_H^+} \check{\Delta}(\alpha).\]
\begin{lem}
 The restriction of $\theta$ to $\check{\Delta}$ is a surjection onto $\Delta_H^+$.
 \end{lem}

 \begin{proof}
 Indeed, Theorem~\ref{thm.embedding} ensures that every root of~$\Delta_H$ is the image under $\theta$ of some root of~$\Delta_G$.
Because $r \colon \check{\b}^{\rm solv} \rightarrow \operatorname{ad}_\h \b_H$ is onto, for $\alpha \in \Delta_H^+$, the root space $\lieh_\alpha$ is contained in the
image under $r$ of the nilpotent radical
 of $\check{\b}^{\rm solv}$. In our basis from Lemma~\ref{lem.solvable.in.p}, the latter subalgebra is strictly upper-triangular,
and is therefore contained in the sum of positive root spaces of $\lieg$. Thus for any $X \in \lieh_\alpha$ the unique $\check{X} \in r^{-1}(X) \cap \check{\lieh}_\alpha$ can be written \smash{$\check{X} = \sum_{\beta \in \Delta_G^+} \check{X}_\beta$}. The set of~$\beta$ occurring in this sum is nonempty and contained in $\theta^{-1}(\alpha)$.
\end{proof}

\begin{lem}\label{lem.root.lifts}
Let $i' \in \{ 1, \dots, p' \}$ and assume that $\lieh_{ \alpha_{i'}} \subset \lieh^0$. Then
\begin{enumerate}\itemsep=0pt
\item[$(1)$] roots in $\check{\Delta} (2 \alpha_{i'} )$ are sums of two distinct roots of $\check{\Delta}(\alpha_{i'})$,
\item[$(2)$] there is a terminal or an initial long root in $\check{\Delta}(2 \alpha_{i'})$.
\end{enumerate}
\end{lem}

\begin{proof}
Suppose $\lieh_{ \alpha_{i'}} \subset \lieh^0$ and apply Lemma~\ref{lem.lots.in.h0}\,(2) to obtain $\check{\lieh}_{ \alpha_{i'}} = \iota\bigl(\lieh_{ \alpha_{i'}}\bigr)$. By Proposition~\ref{prop.vincent}\,(1), $\lieh_{2 \alpha_{i'}}$ is also in $\lieh^0$, so $\check{\lieh}_{2 \alpha_{i'}} = \iota\bigl(\lieh_{2 \alpha_{i'}}\bigr)$, as well.

Every nonzero element of $\lieh_{2 \alpha_{i'}}$ can be obtained as $Z = [X,Y]$ for $X,Y \in \lieh_{\alpha_{i'}}$ distinct nonzero elements. Using (\ref{eqn.iota.in.p}), we obtain
\[\check{\lieh}_{2 \alpha_{i'}} = \iota\bigl(\lieh_{2\alpha_{i'}}\bigr) = \iota \bigl( \bigl[ \lieh_{\alpha_{i'}}, \lieh_{\alpha_{i'}}\bigr] \bigr) = \bigl[ \iota\bigl(\lieh_{ \alpha_{i'}}\bigr), \iota\bigl(\lieh_{ \alpha_{i'}}\bigr)\bigr] = \bigl[ \check{\lieh}_{\alpha_{i'}}, \check{\lieh}_{\alpha_{i'}}\bigr].\]
Therefore, $\check{Z} = \bigl[\check{X},\check{Y}\bigr]$, and
point (1) follows.

Part (2) is vacuously true if $p=q$, since all roots in $\lieg$ are initial or terminal.
If $\beta_j$ is a short root in $\check{\Delta}(2 \alpha_{i'})$, then it was obtained as the sum of $\beta_j - \beta_i$ and $\beta_i$ in $\check{\Delta}(\alpha_{i'})$ for some $i > j$; this is the only way to obtain $\beta_j$ as a sum of two positive roots. There are $X,Y \in \lieh_{\alpha_{i'}}$ with $[X,Y]=Z \neq 0$ and
such that
\[
\iota(Y)_{\beta_j - \beta_i} = 0, \qquad
[\iota(Y),\iota(X)]_{\beta_j} = [\iota(Y)_{\beta_i}, \iota(X)_{\beta_j - \beta_i}] = \iota(Z)_{\beta_j} \neq 0.\]
To ensure the first equality, we are using that $\lieg_{\beta_j-\beta_i}$ is one-dimensional, and replacing $Y$ by a~linear combination with $X$ if necessary. The last equality relies on (\ref{eqn.iota.in.p}).
Now
\[
\iota([Y,Z]) = [\iota(Z),\iota(Y)] = 0
\] but $[ \iota(Z)_{\beta_j},\iota(Y)_{\beta_i}] \neq 0$. Since $\theta(\beta_i) = \alpha_{i'}$ and $\theta(\beta_j) = 2 \alpha_{i'}$, the components
$\iota(Y)_{\beta_j}$ and $\iota(Z)_{\beta_i}$ are both $0$. To obtain $\beta_i + \beta_j$ as the sum of two positive roots of $\Delta_G$, these could be $\beta_i$ and $\beta_j$, or an initial plus a terminal root, of the form $(\beta_i - \beta_k) + (\beta_j + \beta_k)$ for some $k > i$, or vice-versa. In order that the nonzero bracket $[ \iota(Z)_{\beta_j},\iota(Y)_{\beta_i}]$ cancels, there must be a nonzero component of $\iota(Z) \in \check{\lieh}_{2 \alpha_{i'}}$ on one of these initial or terminal roots, yielding an initial or terminal root in~$\check{\Delta}(2 \alpha_{i'})$, as claimed.
\end{proof}

When there is equality in Proposition~\ref{prop.sutoso.ineq}, then the root correspondence given by $\theta$ has the following form.

\begin{prop}
\label{prop.sutoso.roots}
Let $H$ be locally isomorphic to $\SU(p',q')$ for $0 < p' < q'$, with Borel subgroup~$B_H$. Let $\lieg = \so(p+1,q+1)$ with $p \leq q$, and $P < G$ the stabilizer of an isotropic line of $\BR^{p+1,q+1}$. Suppose there is an embedding of $(\lieh, \operatorname{Ad} B_H)$ in $(\lieg, \operatorname{Ad}P)$ with $\operatorname{ad}(\iota(\lieh) \cap \p ) \subseteq \check{\b}$.
If~${p=2p'-1}$, then $p<q$, and, up to conjugation in $P$, the set $\check{\Delta}(\alpha_{i'})$ comprises two short roots for each $i'=1, \dots, p'$, and $\overline{\lieh}$ contains a maximal isotropic subspace of $\BR^{p,q}$.
\end{prop}

\begin{proof}
Let $p=1$ and let $\alpha = \alpha_1$ be the simple positive root of $\Delta_H^+$. It is shown on \cite[p.~9]{pecastaing.rank1} that $\b_H \subset \lieh^0$.
Lemma~\ref{lem.root.lifts}\,(2) says
there is an initial or terminal root in $\check{\Delta}(2 \alpha)$. In rank 2, the initial root of $\Delta_G^+$ cannot be obtained as a sum of two positive roots, so $\check{\Delta}(2 \alpha)$ must contain the terminal root. By Lemma~\ref{lem.root.lifts}\,(1), $\{ \beta_1, \beta_2 \} \subseteq \check{\Delta}(\alpha)$. In particular, $\theta(\beta_1) = \theta(\beta_2) = \alpha$, which determines $\theta$. We conclude $\check{\Delta}(\alpha) = \{ \beta_1, \beta_2 \}$.

Next let $p \geq 3$.
If $\lieh_{\alpha_{i'}} \subset \h^0$ for all $i'$, then, as seen above multiple times, Lemma~\ref{lem.lots.in.h0}\,(1) and (\ref{eqn.iota.in.p}) lead to a homomorphism $\h \rightarrow \q_{ss} \cong \so(p,q)$ and thence, up to isogeny of $H$, to an~embedding $(\h, {\rm Ad}_\h B_H)$ into $(\g, {\rm Ad}_\g P)$, by Theorem~\ref{thm.embedding}. Proposition~\ref{prop.sutoso.ineq} then gives ${p' \leq p/2 < (p+1)/2}$, contradicting the hypothesis of equality. Thus there is $i'$ such that $\overline{\lieh}_{\alpha_{i'}} \neq 0$.
If $p=3$ and $p' = 2$, then we will now show that $\overline{\lieh}_{ - \alpha_{j'}} = 0$ $\forall j' \neq i'$.

Suppose that both $\overline{\lieh}_{- \alpha_1}$ and $\overline{\lieh}_{- \alpha_2}$ are nonzero. By Lemma~\ref{lem.about.h0}\,(2),
\[ \lieh_{- \alpha_1 - \alpha_2} \cap \lieh^0 = \lieh_{- 2 \alpha_{k'}} \cap \lieh^0 = 0 , \qquad k'=1,2\]
An $[s]$-isotropic subspace has dimension $\leq 3$, so Lemma~\ref{lem.h.mod.h0}\,(1) implies $\theta(\beta_1) = \alpha_{1} + \alpha_{2}$. A~maximal isotropic subspace for $[s]$ is \smash{$\overline{\lieh}_{- 2 \alpha_1} + \overline{\lieh}_{- \alpha_1} + \overline{\lieh}_{ - \alpha_2}$}, and all three summands are one-dimensional. Then $\lieh_{- \alpha_{k'}} \cap \lieh^0 \neq 0$ for $k'=1,2$. From the $\ad \b_H$-invariance of $\lieh^0$, we obtain $\a_H \subset \lieh^0$ and therefore $\lieh_{\alpha} \subset \lieh^0$ for all $\alpha \in \Delta_H^+$ by Lemma~\ref{lem.about.h0}\,(1). Using Lemma~\ref{lem.lots.in.h0}, we will henceforth assume $\check{\lieh}_{\alpha} = \iota(\lieh_\alpha)$ for all $\alpha \in \Delta_H^+$.

For $k'=1,2$, let
\[I_{k'} = \bigoplus_{i'=1,2} \lieh_{- \alpha_{i'}} \oplus \lieh_{- 2 \alpha_{k'}}, \]
The projections \smash{$\overline{\iota(I_{k'})}$} are maximal isotropic subspaces of $\lieg/\p$ for $k'=1,2$, but they are not mutually orthogonal.
 The weights for $\a_G$ on $\lieg/\p \cong \BR^{p,q}$ are
\[\Sigma = \{ - \beta_1 \pm \beta \mid \beta = 0, \beta_2, \beta_3, \beta_4 \},\]
if $p=q$, then $\beta=0$ does not occur.
The projected root spaces $\overline{\lieh}_\alpha$ for $\a_H$ have images \smash{$\overline{\iota(\lieh_{\alpha})}$} which decompose into the sum of their intersections with the weight spaces of $\check{\a}$ in $\lieg/\p$.
Since $\theta(- \beta_1) = - \alpha_1 - \alpha_2$ and $\overline{\lieh}_{\alpha_{k'}}$ is $1$-dimensional for $k'=1,2$, there are $i,j > 1$ such that
\[ \overline{\iota(\lieh_{-\alpha_1})} \subset \overline{\lieg}_{- \beta_1 \pm \beta_i} \qquad \mbox{and} \qquad \overline{\iota(\lieh_{-\alpha_2})} \subset \overline{\lieg}_{- \beta_1 \pm \beta_j}.\]
Therefore, $\theta(\beta_i) = \pm \alpha_2$ and $\theta(\beta_j) = \pm \alpha_1$; in particular, $i \neq j$.

The images \smash{$\overline{\iota\bigl(\lieh_{- 2 \alpha_{k'}}\bigr)}$} do not meet $\overline{\lieg}_{- \beta_1}$ for $k'=1,2$, also because $\theta(- \beta_1) = - \alpha_1 - \alpha_2$.
A~maximal isotropic subspace of
 $ \bigoplus_{i > 1} \overline{\lieg}_{- \beta_1 \pm \beta_i}$ is of the form
 \begin{equation*}
 \biggl( \bigoplus_{k \in A} \overline{\lieg}_{- \beta_1 + \beta_k} \biggr) \oplus \biggl( \bigoplus_{\ell \in B} \overline{\lieg}_{- \beta_1 - \beta_\ell} \biggr)
\end{equation*}
 for $A \sqcup B = \{ 2, 3,4 \}$, since the scalar product gives a nondegenerate pairing of each $\overline{\lieg}_{- \beta_1 - \beta_k}$ with $\overline{\lieg}_{- \beta_1 + \beta_k}$. There is therefore $k \neq 1, i, j$ such that
 \[ \overline{\iota\bigl(\lieh_{- 2 \alpha_{k'}}\bigr)} \subset \overline{\lieg}_{- \beta_1 - \beta_k} \qquad \mbox{and} \qquad \overline{\iota\bigl(\lieh_{- 2 \alpha_{\ell'}}\bigr)} \subset \overline{\lieg}_{- \beta_1 + \beta_k}, \qquad \bigl\{ k', \ell' \bigr\} = \{ 1,2 \}.\]
 Now all values of $\theta$ are determined, with $\theta(\beta_k) = \alpha_{k'} - \alpha_{\ell'}$.
Note that because $\overline{\lieh}_{- \alpha_1 - \alpha_2} \neq 0$ and~$- \beta_1$ is the only possible weight in $\Sigma$ mapping to $- \alpha_1 - \alpha_2$, it must indeed be a weight, and~${p < q}$.

By Lemma~\ref{lem.root.lifts}\,(2), each $\check{\Delta}(2 \alpha_{m'})$, for~${m'=1,2}$, contains an initial or terminal root. From the determined values of $\theta$,
\[ \beta_1 + \beta_k \in \check{\Delta}(2 \alpha_{k'}), \qquad \beta_1 - \beta_k \in \check{\Delta}(2 \alpha_{\ell'}) \]
By Lemma~\ref{lem.root.lifts}\,(1), each of these roots is obtained as a sum of two distinct roots of $\check{\Delta}(\alpha_{m'})$ for~${m'=1,2}$. In particular,
\[\beta_1 - \beta_k = (\beta_1 - \beta_i) + (\beta_i - \beta_k) = (\beta_1 - \beta_j) + (\beta_j - \beta_k)\]
are the two possibilities for the summands from $\check{\Delta}(\alpha_{\ell'})$. The value $\theta(\beta_1 - \beta_i) = \alpha_{\ell'}$
only if~${\theta(\beta_i) = \alpha_2}$ and $\ell'=1$;
then also $\theta(\beta_j- \beta_k) = \alpha_1$. The alternative $\ell'=2$ is equally possible provided $\theta(\beta_j) = \alpha_1$. We will assume $\ell'=1$ for the remainder of the proof, which is easily adapted for $\ell'=2$.

Next, there are two possible decompositions into summands from $\check{\Delta}(\alpha_{2})$:
\[ \beta_1 + \beta_k = (\beta_1) + (\beta_k) = (\beta_1 - \beta_j) + (\beta_j + \beta_k)\]
but since $\theta(\beta_1), \theta(\beta_k) \neq \alpha_2$, the only possibility is $\beta_1 - \beta_j, \beta_j + \beta_k \in \check{\Delta}(\alpha_2)$. Thus $\theta(\beta_j) = \alpha_1$.

Observe that $\check{\Delta}(2 \alpha_2) = \{ \beta_1 + \beta_k \}$, as there are no other elements of $\Delta_G^+$ in $\theta^{-1}(2 \alpha_2)$. Let~${X \in \lieh_{2 \alpha_2}}$, so $\iota(X) \in \lieg_{\beta_1 + \beta_k}$. Let~${Y \in \lieh_{\alpha_1}}$ with $\iota(Y)_{\beta_i - \beta_k} \neq 0$. Then
\[[\iota(X), \iota(Y)]_{\beta_1 + \beta_i} = [\iota(X)_{\beta_1 + \beta_k} , \iota(Y)_{\beta_i - \beta_k}] \neq 0.\]
Therefore, $[\iota(X), \iota(Y)] \neq 0$, which implies by (\ref{eqn.iota.in.p}) and injectivity of $\iota$ that $[X,Y] \neq 0$. But this is a~contradiction because $2 \alpha_2 + \alpha_1 \notin \Delta_H$. This completes the proof of the claim for~${p=3}$,~${p'=2}$.

For $p > 3$, Lemma~\ref{lem.h.mod.h0}\,(1) gives the conclusion as above that $\lieh_{- \alpha_{j'}} \subset \lieh^0$ for all $j'\neq i'$. The $\ad \b_H$-invariance of \smash{$\lieh^0$} from Proposition~\ref{prop.vincent}\,(1) implies that \smash{$\lieh_{ \alpha_{j'}} \subset \lieh^0$} for all $j' \neq i'$. By Lemma~\ref{lem.about.h0}\,(3), \smash{$\lieh_{- \alpha_{i'}} \cap \lieh^0 = 0$}.
Now
let $\theta(2 \beta_1) = 2 \alpha_{i'} + \alpha'$
as in Lemma~\ref{lem.h.mod.h0}\,(2).

\textit{Case $1$}: $\theta(\beta_1) = \alpha_{i'}$. As seen above,
root spaces are $[s]$-orthogonal unless the roots add up to the conformal factor, in this case $- 2 \alpha_{i'}$. Therefore, the subspace \smash{$\bigoplus_{j'} {\lieh}_{- \alpha_{i'} - \alpha_{j'}}$} projects modulo~$\lieh^0$ to an isotropic subspace for $[s]$; by Lemma~\ref{lem.about.h0}\,(2), it has trivial intersection with~$\lieh^0$, so the projection has dimension $2p'-1 = p$.
Any \smash{$\overline{\lieh}_{\alpha_{i'} - \alpha_{j'}}$}, with $j' \neq i'$, could be added to make a~larger isotropic subspace; they must therefore be $0$ -- that is $\lieh_{\alpha_{i'} - \alpha_{j'}} \subset \lieh^0$ for all $j' \neq i'$.
 Because~\smash{${\bigl[\lieh_{ \alpha_{j'}}, \lieh_{\alpha_{i'} - \alpha_{j'}} \bigr] = \lieh_{\alpha_{i'} }}$},
Proposition~\ref{prop.vincent}\,(1) implies \smash{$\lieh_{\alpha_{i'}} \subset \lieh^0$} as well.

Denote by
\[J = \lieh_{- \alpha_{i'}} \oplus \bigoplus_{j'} \lieh_{- \alpha_{i'} - \alpha_{j'}} \qquad \mbox{and} \qquad I = \bigoplus_{j'} \lieh_{- \alpha_{i'} - \alpha_{j'}} .\]
As above,
the image \smash{$\overline{\iota(J)}$} is a sum of $\check{\a}$-weight spaces in $\lieg/\p$, and
the image \smash{$\overline{\iota(I)}$} is a maximal isotropic subspace of $\lieg/\p \cong \BR^{p,q}$.
The weights for $\a_G$ on $\lieg/\p$, for $p < q$, are
\[\Sigma = \{ - \beta_1 \pm \beta \mid \beta = 0, \beta_2, \dots, \beta_{p+1} \},\]
and when $p=q$ then $\beta=0$ does not occur.
Let $\Sigma_I$ comprise the roots in $\Sigma$ for which \smash{$\overline{\iota(I)}$} has nontrivial projection on the corresponding $\a_G$-weight space. Elements of $\Sigma_I$ must have the additional property that $\theta(- \beta_1 \pm \beta) = - \alpha_{i'} - \alpha_{j'}$ for some $j'$; therefore, $- \beta_1 \notin \Sigma_I$.
 As in the previous proof, a maximal isotropic subspace of
 $ \bigoplus_{i > 1} \lieg_{- \beta_1 \pm \beta_i}$ is of the form
 \begin{equation}
 \label{eqn.iso.subspace}
 \biggl( \bigoplus_{i \in A} \lieg_{- \beta_1 + \beta_i} \biggr) \oplus \biggl( \bigoplus_{j \in B} \lieg_{- \beta_1 - \beta_j} \biggr)
\end{equation}
 for $A \sqcup B = \{ 2, \dots, p+1 \}$
 The corresponding $\Sigma_I$ is $\{ - \beta_1 + \beta_i \}_{i \in A}$ union $\{ - \beta_1 - \beta_i \}_{i \in B}$.
Thus all values of $\theta$ are determined; in particular $\theta(\beta_i) \neq 0$ for all $i > 1$.

Now
\smash{$\overline{\iota\bigl(\lieh_{- \alpha_{i'}}\bigr)} \subseteq \overline{\lieg}_{- \beta_1}$} because $\theta(- \beta_1 \pm \beta_i) \neq - \alpha_{i'}$ for any $i$; in particular, $p < q$ in this case.
By (\ref{eqn.iota.in.p}),
\[\bigl[ \check{\lieh}_{\alpha_{j'}}, \iota\bigl(\lieh_{- \alpha_{i'} - \alpha_{j'}}\bigr)\bigr]= \bigl[ \iota\bigl(\lieh_{\alpha_{j'}}\bigr), \iota\bigl(\lieh_{- \alpha_{i'} - \alpha_{j'}}\bigr)\bigr] = \iota\bigl(\lieh_{- \alpha_{i'}}\bigr),\]
 which means that for each $- \beta_1 \pm \beta_i \in \Sigma_I$, there is some $j'$ and there is $\beta' \in \check{\Delta}(\alpha_{j'})$ such that $\pm \beta_i + \beta' = 0$. As $\beta'$ is a positive root, we conclude $\Sigma_I = \{ - \beta_1 - \beta_i \mid i > 1 \}$, and \smash{$\overline{\iota(I)} = \bigoplus_{j > 1} \overline{\lieg}_{- \beta_1 - \beta_j}$}. 
Because $\overline{\lieh}_{- \alpha_{i'} - \alpha_{j'}}$ with $j' \neq i'$ is two-dimensional while each $\overline{\lieg}_{- \beta_1 - \beta_i}$ is one-dimensional, there~are at least two distinct $\beta_j$ with $\rho(\beta_j) = \alpha_{j'}$ for each $j' \neq i'$. There is moreover some $i \neq 1$ such~that $\theta(\beta_i) = \alpha_{i'}$, and such that \smash{$\overline{\iota(\lieh_{- 2 \alpha_{i'} })} = \overline{\lieg}_{- \beta_1 - \beta_i}$}. Since ${p' = (p+1)/2}$, there are exactly two indices $j$ such that $\theta(\beta_j) = \alpha_{j'}$ for all $j'$. The conclusion in this subcase follows.

\textit{Case $2$}: $\theta(2 \beta_1) = 2 \alpha_{i'} + \alpha'$ with $\alpha' = \alpha_{j'}$ or $\alpha_{j'} + \alpha_{\ell'}$ for some $j' \neq i'$, possibly with $j' = \ell'$. In this case $\overline{\lieh}_{- \alpha_{i'}}$ is $[s]$-isotropic, and belongs to an isotropic subspace together with \smash{$\bigoplus_{j'\neq k'} \overline{\lieh}_{- \alpha_{i'} - \alpha_{k'}}$}. By Lemma~\ref{lem.about.h0}\,(2), each \smash{$\lieh_{- \alpha_{i'} - \alpha_{k'}}$} has trivial intersection with \smash{$\lieh^0$}; also \smash{$\overline{\lieh}_{- \alpha_{i'}}$} has dimension $2$. As in the previous subcase, any \smash{$\overline{\lieh}_{\alpha_{i'} - \alpha_{k'}}$} for $k' \neq i'$ could be added to make a~larger isotropic subspace, so these root spaces are contained in $\lieh^0$, and therefore also \smash{$\lieh_{\alpha_{i'}} \subset \lieh^0$}.

Now let $J$ be as above, while
\[I = \lieh_{- \alpha_{i'}} \oplus \bigoplus_{j'\neq k'} \lieh_{- \alpha_{i'} - \alpha_{k'}}.\]
As before,
\smash{$\overline{\iota(I)}$} is a maximal isotropic subspace of $\lieg/\p \cong \BR^{p,q}$. The subset $\Sigma_I \subset \Sigma$ is defined as above.
 Since $\theta(- \beta_1)$ is not a root occurring in $I$, then $- \beta_1 \notin \Sigma_I$. Thus \smash{$\overline{\iota(I)}$} has the same form as (\ref{eqn.iso.subspace}). By dimension considerations, the restriction of $\theta$ to $\Sigma_I$ is a 2-to-1 map to the roots of $\Delta_H$ occurring in $I$, except to $- 2 \alpha_{i'}$, onto which it is 1-to-1. The values of $\theta(\beta_i)$ are thus determined for all $i$.

Now $ \check{\Delta}(\alpha_{i'})$ has the form $\{ \beta_1 +\epsilon \beta_i, \beta_1 + \epsilon' \beta_j \}$ for some $i \neq j$ different from $1$ and $\epsilon, \epsilon' = \pm 1$. But $\check{\Delta}(2\alpha_{i'})$ also comprises a root of the form $\beta_1 \pm \beta_k$ for $k \neq 1$.
 The two distinct elements of $\check{\Delta}(\alpha_{i'})$ do not sum to the element of $\check{\Delta}(2\alpha_{i'})$
contradicting Lemma~\ref{lem.root.lifts}\,(1). This subcase is eliminated and the proof is finished.
\end{proof}

\section{Embedding theorem in the conformal tractor bundle}

In this section, we work with the conformal tractor bundle to complete the proof of Theorem~\ref{thm.conf.application}.
Therefore, we let $H$, $M$, and $(\lieg,P)$ be as in the theorem, and assume $p = 2p'-1$.

\subsection{Local conformal flow with balanced isotropy}

By Proposition~\ref{prop.sutoso.roots}, there is a partition of the short roots $\beta_1, \dots, \beta_{p+1}$ of $\Delta_G^+$ into $p'$ pairs $\bigl(\beta_{k_{i'}}, \beta_{\ell_{i'}}\bigr)$, $i' = 1, \dots, p'$, such that $\theta\bigl(\beta_{k_{i'}} - \beta_{\ell_{i'}}\bigr) = 0$ for each $i'$. As $\check{\a}$ is $p'$-dimensional and $p' = (p+1)/2$, a basis for $\check{\a}$ is \smash{$\bigl\{ \check{Y}_{k_{i'}} + \check{Y}_{\ell_{i'}} \mid i' = 1, \dots, p' \bigr\}$}, where $\check{Y}_j$ is the coroot of $\a_G$ dual to~$\beta_j$ for each $j$. Let $i'$ be such that $k_{i'} = 1$. By applying an inner automorphism of $P$ if necessary, we may assume $\ell_{i'} = 2$. Now set $\check{Y} = \check{Y}_1 + \check{Y}_2 \in \check{\a}$.

The one-parameter subgroup
\[\big\{ h^t = {\rm e}^{t \check{Y}} \big\} = {\rm diag}\bigl(e^{t}, {\rm e}^{t}, 1, \dots, 1, {\rm e}^{-t}, {\rm e}^{-t}\bigr) < P,\]
and it belongs to $\check{B}^{\rm solv}$. Recall that $\check{B}^{\rm solv}$ was obtained from the embedding theorem for a~particular $x_0 \in M$ and $\hat{x}_0 \in \hat{M}_{x_0}$. By Corollary~\ref{cor.emb.frob}, there is a one-parameter subgroup ${\bigl\{ \varphi^t_Y \bigr\} < \Aut^{\rm loc}_{x_0}(M,\omega)}$ with isotropy $\bigl\{ h^t \bigr\}$ with respect to $\hat{x}_0$.
The representation of $\bigl\{ h^t \bigr\}$ on~$\BR^{p,q}$ is ${\rm diag}\bigl(1, {\rm e}^{-t}, \dots, {\rm e}^{-t}, {\rm e}^{-2t}\bigr)$, and is said to be of \emph{balanced linear} type because the dilation and hyperbolic isometric components balance each other.

\subsection{Isotropic tractor solutions}

Existence of a flow with the dynamics of $\bigl\{ \varphi^t_Y \bigr\}$ does not necessarily imply local conformal flatness~-- see Alekseevsky's important examples in~\cite{alekseevski.selfsim}~-- but it does impose strong restrictions on the Weyl curvature. These in turn lead to local tractor solutions.

For $v \in \BR^{p,q}$ an isotropic vector, the \emph{annihilator of the flag} $\BR v \subset v^\perp \subset \BR^{p,q}$ comprises the nilpotent subalgebra of $\so(p,q)$ which preserves the flag and is zero on the associated graded space $\BR v \oplus \bigl( v^\perp/\BR v \bigr) \oplus \bigl( \BR^{p,q}/v^\perp \bigr)$; it will be denoted $\n(v)$.

Let $\mathcal{V}$ be the standard conformal tractor bundle associated to the representation $\BR^{p+1,q+1}$ of~${P < {\rm SO}(p+1,q+1)}$, with standard conformal tractor connection $\nabla = \nabla^\omega$. Denote by $\mathcal{S}$ the $\nabla$-parallel sections of $\mathcal{V}$, and recall from Section~\ref{sec.conf.tractor} that they correspond to almost-Einstein rescalings of a metric in the given conformal class on $M$.

The following consequences of balanced linear isotropy are already known in the conformal Lorentzian case (see \cite[Lemma~6.5 and Theorem~1.3\,(2)]{frances.degenerescence}, \cite[Section~4]{mp.confdambra}, \cite[Proposition~5.1]{dfmpz.balanced}),
and generalize to higher signature as follows.

\begin{prop}[compare \cite{dfmpz.balanced,frances.degenerescence,mp.confdambra}]
\label{prop.balanced.weyl}
Let $(M,g)$ be a pseudo-Riemannian manifold, and suppose that the stabilizer of $x_0 \in M$ contains a conformal flow \smash{$\bigl\{ \varphi^t_Y \bigr\}$} with isotropy generated by $\check{Y} = {\rm diag}(1,1, 0 , \dots, 0, -1,-1 )$ with respect to $\hat{x}_0 \in \pi^{-1}(x_0)$. Then there are
\begin{enumerate}\itemsep=0pt
\item[$(1)$] a Ricci-flat metric $g_0 \in [ g |_U]$ in a neighborhood $U$ of $x_0$, and
\item[$(2)$] an isotropic vector field $X \in \mathcal{X}(U)$ which is parallel for the Levi-Civita connection of $g_0$, such that
\item[$(3)$] the Weyl tensor $W_x$ has values in $\n(X_x)$ for all $x \in U$.
\item[$(4)$] Together these yield a $2$-dimensional totally isotropic space $\mathcal{I}_U$ of parallel sections of $ \mathcal{V} |_U$, having trivial intersection with $\mathcal{V}^0$.
\end{enumerate}
If $(M,g)$ is analytic and simply connected, then the solutions in $\mathcal{I}_U$ extend to $M$, comprising a~subspace $\mathcal{I} \subset \mathcal{S}$, and $X$ extends to an isotropic conformal vector field on $M$.
\end{prop}

\begin{proof}
As above, consider the one-parameter subgroup
\[\bigl\{ h^t \bigr\} = {\rm diag}\bigl({\rm e}^{t}, {\rm e}^{t}, 1, \dots, 1, {\rm e}^{-t}, {\rm e}^{-t}\bigr) < P,\]
the isotropy of $\bigl\{ \varphi^t_Y \bigr\}$ with respect to $\hat{x}_0$.
On $\BR^{p,q}$, denote by $\{ e_1, \dots, e_n \}$ the corresponding null-orthonormal frame, where $n=p+q$, in which $\bigl\{h^t \bigr\}$ is
\[{\rm diag}\bigl(1, {\rm e}^{-t}, \dots, {\rm e}^{-t}, {\rm e}^{-2t}\bigr).\]
 In this frame, the Weyl tensor $W_{x_0}$ is a homomorphism $\hat{W}_{\hat{x}_0}\colon \wedge^2 \BR^{p,q} \rightarrow \so(p,q)$, which is $\{ h^t\}$-equivariant. The eigenvalues of $\Ad h^t$ on $\so(p,q)$ are ${\rm e}^t$, $1$, and ${\rm e}^{-t}$, and the eigenspace for ${\rm e}^{-t}$ is precisely the nilpotent subalgebra $\n(e_n)$. It follows that $\hat{W}_{\hat{x}_0} (u,v) \in \n(e_n)$, for all $u,v \in \BR^{p,q}$.

The flow $\bigl\{ h^t \bigr\}$ is \emph{stable} for $t \rightarrow \infty$ in the sense of \cite[Definition~4.2]{frances.degenerescence}. It follows that~$\bigl\{ h^t \bigr\}$ remains a \emph{holonomy} flow for points nearby $x_0$, meaning that there is a neighborhood $U$ of~$x_0$ with a section $\hat{U} \subset \hat{M}$ containing $\hat{x}_0$ such that $\varphi^t_Y. \hat{x} . h^{-t}$ is bounded in $\hat{M}$ as $t \rightarrow \infty$ for all~${\hat{x} \in \hat{U}}$ (see~\cite[Proposition~2.8]{cap.me.parabolictrans}). The section $\hat{U}$ can be taken to be $\exp_{\hat{x}_0}(\g_-)$, restricted to a~sufficiently small neighborhood of $0 \in \g_-$; it is in particular real-analytic.

For any $\hat{x} \in \hat{U}$, the Weyl curvature $\hat{W}_{\hat{x}}$ must remain bounded -- though not necessarily invariant -- under the action of $\bigl\{ h^t \bigr\}$. It follows that $\hat{W}_{\hat{x}} (u,v) \in \n(e_n)$ for all $u,v \in \BR^{p,q}$, $\hat{x} \in \hat{U}$.
Because it is \smash{$\bigl\{ \varphi^t_Y \bigr\}$}-invariant, the full Cartan curvature $\kappa_{\hat{x}}$ is bounded under the action of the holonomy \smash{$\bigl\{ h^t \bigr\}$}. The $\p_+$-component of $\kappa_{\hat{x}}$ is in $\wedge^2 \BR^{p,q*} \otimes \BR^{p,q*}$, which has no stable elements under \smash{$\bigl\{ h^t \bigr\}$}. Consequently, $\kappa_{\hat{x}} = \hat{W}_{\hat{x}}$ for all $\hat{x} \in \hat{U}$.

Let $\mathcal{R}_U$ be the saturation $\hat{U} \cdot J$, for $J \cong J_0 \ltimes J_1 < P$, where $J_0 < {\rm CO}(p,q)$ is the stabilizer of $\BR e_n$ and $J_1 < P^+ \cong \BR^{p,q*}$ is the annihilator of $e_n^\perp$.
This saturation is an analytic reduction of $\hat{M}$ to $J$.
Observe that the Lie algebra{\samepage
\[\mathfrak{j}_1 = \{ \xi \in \p^+ \mid [\xi, \lieg_{-1}] \subset \mathfrak{j}_0 \}.\]
Thus $\g_- + \mathfrak{j}$ is a subalgebra of $\g$.}

Considering $\n(e_n)$ inside $\p$, it is normalized by $J$. Thus on $\mathcal{R}_U$ the Cartan curvature retains the properties of being equal to the Weyl curvature and of having values in $\n(e_n)$. It follows from the formula, for $Y,V \in \mathcal{X}(\hat{M})$,
\[\omega_{\hat{x}}[Y,V] = Y.\omega_{\hat{x}}(V) - V. \omega_{\hat{x}}(Y) - \kappa_{\hat{x}}(\omega_{\hat{x}}(Y),\omega_{\hat{x}}(V)) + [ \omega_{\hat{x}}(Y),\omega_{\hat{x}}(V)] \]
that the distributions $\mathcal{D}' = \omega^{-1}(\g_- + \mathfrak{j})$ and $\mathcal{D} = \omega^{-1}(\g_- + \n(e_n))$ are involutive in restriction to~$\mathcal{R}_U$ (see for example \cite[Section~4]{mp.confdambra}). The integral curves of the $\omega$-constant vector fields in~${\g_- + \mathfrak{j}}$ through $\hat{x}_0$ are contained in $\mathcal{R}_U$, by construction.

We claim that $\mathcal{R}_U$ is an integral submanifold for $\mathcal{D}'$. Let $y, v \in \g_-$, so $\exp_{\hat{x}_0} (tv) \in \hat{U}$ for all sufficiently small $t$. To prove the claim, it suffices to show $\omega((\exp_{\hat{x}_0})_{*v} (y)) \in \g_- + \mathfrak{j}$ for any such~$y$,~$v$, with~$v$ sufficiently small. Indeed, setting $\hat{x} = \exp_{\hat{x}_0} (v)$, once \smash{$T_{\hat{x}} \hat{U} \in \mathcal{D}'_{\hat{x}}$}, then it follows from the axioms for $\omega$ that $T_{\hat{x}} \mathcal{R}_U = \mathcal{D}'_{\hat{x}}$, and $T_{\hat{x}.g} \mathcal{R}_U \subset \mathcal{D}'_{\hat{x}.g}$ for all $g \in J$; moreover, every point of $\mathcal{R}_U$ is of the form $\hat{x}.g$ for such $\hat{x} = \exp_{\hat{x}_0}(v)$ and $g \in J$.

Consider a two-parameter map $\sigma(t,s) = \exp_{\hat{x}_0} (t(v+sy))$, so that
\[ \frac{\partial \sigma}{\partial s} (t,0) = (\exp_{\hat{x}_0})_{*tv}(ty) \qquad \mbox{and} \qquad \omega \left( \frac{\partial \sigma}{\partial t} (t,s) \right) = v + sy.\]
The formula for the Cartan curvature gives rise to the ODE
\[ \frac{\mathrm{d}}{\mathrm{d}t} \omega \left( \frac{\partial \sigma}{\partial s}(t,0) \right) = y - \kappa \left( \omega \left( \frac{\partial \sigma}{\partial s} (t,0) \right), v \right) + \left[ \omega \left( \frac{\partial \sigma}{\partial s}(t,0)\right), v \right].\]
This equation is linear modulo $\g_- + \mathfrak{j}$ in the sense that, if $Y \in \g_- + \mathfrak{j}$, then
\[ y - \kappa (Y, v ) + [Y,v] \equiv 0 \qquad \mbox{mod } \g_- + \mathfrak{j}\]
and satisfies the initial condition
\[\omega \left( \frac{\partial \sigma}{\partial s}(0,0) \right) = 0.\]
Therefore,
\[\omega \left( \frac{\partial \sigma}{\partial s}(t,0) \right) \in \g_- + \mathfrak{j} \qquad \forall t,\]
which gives the desired conclusion when $t=1$.

Now the distribution $\mathcal{D}$ is tangent to $\mathcal{R}_U$ and integrable in restriction to it. Let $\mathcal{L}_0$ be the leaf through $\hat{x}_0$ tangent to $\mathcal{D}$. It is a reduction to $U(e_n)$, the unipotent subgroup of $G_0$ with Lie algebra $\n(e_n)$. This reduction corresponds to a metric $g_0 \in [ g|_U]$ which is Ricci-flat. Define $X \in \mathcal{X}(U)$ by \smash{$X_x = \pi_*\bigl(\omega_{\hat{x}}^{-1}(e_n)\bigr)$}, where $\hat{x}$ is the lift of $x$ to $\hat{U}$. Then $X$ is parallel for the Levi-Civita connection of $g_0$. It is thus a Killing field of $g_0$ and in particular a conformal vector field. We have proved (1)--(3).

With respect to the decomposition (\ref{eqn.conf.splitting}), the subspace $\mathcal{I}_U$ contains constant sections of $ \mathcal{V}_{-1} |_U$, corresponding to metrics on $U$ homothetic to $g_0$, as well as $\tilde{X} = X + \sigma_X$, where $\sigma_X$ is a section of $ \mathcal{V}_{-1} |_U$ with gradient equal $X$. Indeed, these sections are seen to be parallel from formula~(\ref{eqn.conf.connxn}) with respect to the metric $g_0$; moreover, their span is two-dimensional and transverse to~$\mathcal{V}_0$.

When $(M,[g])$ is real-analytic, then any local solution lifts and extends to a global solution on the pullback of $\mathcal{V}$ to the universal cover of ${M}$, and the same holds for local conformal vector fields \cite{amores.killing}. Under the assumption that $M$ is simply connected, any local solutions globalize, as does the local conformal vector field $X$. Then the solutions $\mathcal{I}_U$ extend to global solutions $\mathcal{I} \subset \mathcal{S}$, and the extension of $X$ is everywhere isotropic by analyticity, as it is isotropic on $U$.
\end{proof}

\begin{rem}
The conclusions above hold just as well given a nontrivial balanced linear element in the isotropy at a point -- that is, a flow is not needed.
\end{rem}

Here is where the embedding theorem is applied to the conformal tractor bundle:

\begin{prop}
\label{prop.soln.not.fixed}
Let $(M,[g])$ be as in Proposition~$\ref{prop.balanced.weyl}$ and assume that it is real-analytic and simply connected. Let $\mathcal{I} \subset \mathcal{S}$ be the parallel tractors given by Proposition~$\ref{prop.balanced.weyl}$\,$(4)$.
The $H$-action on $\mathcal{S}$ does not fix any nonzero vectors of $\mathcal{I}$.
\end{prop}

\begin{proof}
Let $\hat{x}_0$ be as above, a point at which
\[\bigl\{ h^t = {\rm diag}\bigl({\rm e}^t, {\rm e}^{t}, 1, \dots, 1, {\rm e}^{-t},{\rm e}^{-t}\bigr) \bigr\} < \check{B}^{\rm solv},\]
This point was given by the embedding theorem; it is such that the linear injection $\iota\colon \mathcal{S} \rightarrow {\bf V}$ is the composition of evaluation at $x_0$ with the trivialization $\iota_{\hat{x}_0}$ of $\mathcal{V}_{x_0}$ corresponding to $\hat{x}_0 \in \hat{M}$.

Now $\hat{x}_0$ belongs to the reduction of $\hat{M}$ to $U(e_n)$ constructed in the proof of Proposition~\ref{prop.balanced.weyl} above, corresponding to the Ricci-flat metric $g_0$. The trivialization $\iota_{\hat{x}_0}\colon \mathcal{V}_{x_0} \rightarrow {\bf V}$ therefore identifies the subspaces $(\mathcal{V}_i)_{x_0}$ of the decomposition determined by $g_0$ with ${\bf V}_i$ for $i=1,0,-1$.
The constant sections of $\mathcal{V}_{-1}$, belonging to $\mathcal{I}$, corresponding to metrics homothetic to $g_0$, therefore evaluate under $\iota_{\hat{x}_0}$ to elements of ${\bf V}_{-1} = \BR E_{n+1}$.

The point $\hat{x}_0$ also, by the construction of the isotropic vector field $X$ in the proof of Proposition~\ref{prop.balanced.weyl}\,(2) above, corresponds to a frame of $T_{x_0} M$ in which $X_{x_0}$ evaluates to $e_n \in \BR^{1,n-1} = {\bf V}_0$.
The solution $X + \sigma_X \in \mathcal{I}$ therefore evaluates under $\iota$ to a vector of the form $E_n + b E_{n+1} \in {\bf V}_0 + {\bf V}_{-1}$.
The image $\iota(\mathcal{I})$ is therefore the span of $E_n$ and $E_{n+1}$ in ${\bf V}$. On this subspace, $h^t$~acts as a dilation by ${\rm e}^{-t}$, and has no nonzero fixed vectors.

By the embedding theorem, there is a one-parameter subgroup of $B_H$, the image under~$R$ of~\smash{$\bigl\{ h^t \bigr\}$}, which acts on $\mathcal{I}$ as a dilation by ${\rm e}^{-t}$, in particular nontrivially. (Note that \smash{$\bigl\{ \varphi^t_Y \bigr\}$} may or may not belong to $B_H$; the embedding theorem says in any case that there is a one-parameter subgroup of $B_H$ with representation on $\mathcal{S}$ corresponding to \smash{$\bigl\{ h^t \bigr\}$}.)
\end{proof}

\subsection{Conclusion}

 Let $0 \neq \chi \in \mathcal{I}$. The orbit $H.\chi \subset \mathcal{S}$ spans an irreducible $H$-representation $\mathcal{S'}$,
 which has dimension at least $p'+q' > 2 p' = p+1$, so $\mathcal{S'}$ cannot be a totally isotropic subspace for the tractor form on $\mathcal{S}$. Because the $H$-representation on $\mathcal{S}'$ is irreducible and preserves the scalar product on the tractor bundle, $\mathcal{S}'$ is necessarily a nondegenerate subspace of $\mathcal{S}$. The representation therefore corresponds to a homomorphism, which we may assume is virtually faithful by Proposition~\ref{prop.soln.not.fixed}, $I \colon H \rightarrow L = \SO(\ell, m)$, where $\ell \leq \mbox{min}\{ m, p+1 \}$ and $m \leq q+1$.

Now $H$ acts conformally on $\mbox{\bf M\"ob}^{\ell-1, m-1}$. By Theorem~\ref{thm.embedding} and amenability of $B_H$, there is
an embedding of $(\lieh, B_H)$ into $(\mathfrak{o}(\ell, m), N)$, where $N$ is the stabilizer of an isotropic line of $\mathcal{S}'$, for which we can take $\check{B} = \overline{I(B_H)}_d < N$.
By Proposition~\ref{prop.sutoso.roots}, $\ell \geq 2 p' = p+1$, therefore $\ell = p+1$.

The conformal tractor curvature $R^\omega$ annihilates $\mathcal{S}$, and in particular $\mathcal{S}'$, a nondegenerate subspace for the tractor scalar product of maximal index. The conformal tractor curvature has values $R^\omega_{\hat{x}}(u,v)$ in the subalgebra of $\g \cong \so(p+1,q+1)$ annihilating $\iota_{\hat{x}} (\mathcal{S}')$, for all $\hat{x} \in \hat{M}$ and~${u,v \in T_x M}$, and this subalgebra is compact.
 Now consider the metric $g_0$ on the neighborhood~$U$ from Proposition~\ref{prop.balanced.weyl}, and the expression (\ref{eqn.conf.curv}) for $R^\omega$ with respect to $g_0$. In this expression, according to the decomposition (\ref{eqn.conf.splitting}), $R^\omega$ annihilates ${\bf V}_{1}$ and ${\bf V}_{-1}$ and restricts on~${\bf V}_0$ to the Weyl tensor $W$. The Weyl tensor therefore has values in a compact subalgebra of $\so({\bf V}_0) \cong \so(p,q)$ annihilating a nondegenerate subspace of maximal index.
On the other hand, by Proposition~\ref{prop.balanced.weyl}, the Weyl tensor values $W_x$ belong to a nilpotent subalgebra $\n(X_x)$ for any~${x \in U}$.
Since the intersection of such a nilpotent subalgebra with any compact subalgebra is trivial, we conclude that $W$ vanishes on $U$, and by analyticity, that $W$ vanishes on all of $M$. Thus $(M,[g])$ is conformally flat.

Conformal flatness means that there is a conformal local diffeomorphism, a \emph{developing map}, from the universal cover of $M$, in this case $M$ itself, to $\mbox{\bf M\"ob}^{p,q}$. Since $M$ is also assumed closed, this local diffeomorphism is a covering map. The global conclusion follows, and Theorem~\ref{thm.conf.application} is proved.

\subsection*{Acknowledgements}
We thank the anonymous referees for their comments and suggestions, which significantly helped us to improve presentation of our results.
Some of the work on this paper was carried out during the program ``Twistor theory'', which was partially supported by the EPSRC grant EP/Z000580/1, at the Isaac Newton Institute for Mathematical Sciences, whom we thank for their hospitality. The first author was partially supported during work on this project by the US National Science Foundation award DMS-2109347. The first author affirms her commitment to diversity, equity, and inclusion in mathematics and all sciences and thanks the US National Science Foundation for its support, which has been instrumental at nearly every stage of her career. She salutes the dedicated staff of the Division of Mathematical Sciences in particular. The second author acknowledges support by the grants GA22-00091S from Czech Science Foundation (GA\v CR) and MUNI/R/1435/2024 from the grant agency of Masaryk University (GAMU), as well as by the COST Action CaLISTA CA21109 (\url{https://www.cost.eu}).

\pdfbookmark[1]{References}{ref}
\LastPageEnding

\end{document}